\documentclass[10pt]{article}
\usepackage{graphicx} 
\usepackage{amsmath}
\usepackage{amsthm}
\usepackage{amssymb}
\usepackage{geometry}
\usepackage{esint}
\usepackage{hyperref}
\usepackage{mathrsfs}
\usepackage{graphicx}
\usepackage{color}
\usepackage{verbatim}
\usepackage{booktabs}
\usepackage{siunitx}
\usepackage{float}
\usepackage{multirow}
\usepackage{algorithm}
\usepackage{algpseudocode}
\usepackage{float}
\usepackage{threeparttable}
\usepackage{booktabs}
\theoremstyle{definition}
\newtheorem{theorem}{Theorem}[section]

\newtheorem{corollary}{Corollary}[section]
\newtheorem{lemma}{Lemma}[section]
\newtheorem{remark}{Remark}[section]
\newtheorem{example}{Example}[section]

\newcommand{\RR}{\mathbb{R}}

\renewcommand{\d}{\,\mathrm{d}}

\newcommand{\p}{\partial}
\numberwithin{equation}{section}

\DeclareMathOperator{\supp}{\mathrm{supp}\,}

\title{Identification of a Point Source in the Heat Equation from Sparse Boundary Measurements\thanks{The work of B. Jin is supported by Hong Kong RGC General Research Fund (Projects 14306423 and 14306824),
ANR / RGC Joint Research Scheme (A-CUHK402/24) and a start-up fund from The Chinese University of Hong
Kong. The work of Y. Kian is supported by the French National Research Agency ANR and Hong Kong RGC Joint
Research Scheme for the project IdiAnoDiff (grant ANR-24-CE40-7039).}}
\author{Fangyu Gong\thanks{Department of Mathematics, The Chinese University of Hong Kong, Shatin, N.T., Hong Kong, P. R. China. (\texttt{1155191587@link.cuhk.edu.hk, b.jin@cuhk.edu.hk, 1155228869@link.cuhk.edu.hk})} \and Bangti Jin\footnotemark[2] \and Yavar Kian\thanks{Univ Rouen Normandie, CNRS, Normandie Univ, LMRS UMR 6085, F-76000 Rouen, France (\texttt{yavar.kian@univ-rouen.fr})} \and Sizhe Liu\footnotemark[2]}
\date{}

\begin{document}

\maketitle

\begin{abstract}
In this work we investigate the inverse problem of recovering one point source in the heat equation from sparse boundary measurement, i.e., the flux data at several points on the boundary. We prove the unique recovery of the location and piecewise constant in time amplitude when the domain is the unit ball in $\mathbb{R}^d$ ($d\geq2$), and the unique recovery of the location and compactly supported amplitude when the domain is simply connected, smooth and bounded in $\mathbb{R}^2$, under mild conditions on the observational points. The proof combines distinct analytical tools, including the representation of the flux data via Laplacian eigenfunctions on the unit ball, a detailed analysis of the properties of the heat and Poisson kernels, as well as methods drawn from complex analysis. Further we present several numerical experiments to illustrate the feasibility of the recovery from sparse boundary data.
\end{abstract}

\section{Introduction}
Let $\Omega \subset \mathbb{R}^d$ ($d\ge 2$) be an open bounded, simply connected domain with a smooth boundary $\partial\Omega$. Consider the following initial boundary value problem for the function $u$:
\begin{equation}\label{eqn:heat}
    \left\{\begin{aligned}
    \partial_tu-\Delta u &=F,\quad \mbox{in }\Omega\times (0,T),\\
    u&=0,\quad \mbox{on }\partial\Omega\times(0,T),\\
    u & = 0,\quad \mbox{in } \Omega\times\{0\},
\end{aligned}\right.
\end{equation}
where the source $F(x,t)$ takes the form
\begin{equation}\label{eqn:source}
F(x,t)=g(t)\delta(x-p),\quad \mbox{with } g(t)\in L^2(0,T).
\end{equation}
This equation describes the diffusion of heat in the domain $\Omega$. In this work,
we study the inverse source problem of recovering the information about the source $F$ (i.e., the location $p$ and amplitude $g(t)$) from the observation of the flux $\partial_{\nu}u$ at finitely many points $\{z_\ell\}_{\ell=1}^L$ on the boundary $\partial\Omega$ along the time interval $(0,T)$. The boundary measurements at finitely many points are called sparse. Inverse source problems for the heat equation \eqref{eqn:heat} are motivated by various application in different scientific fields as well as industrial applications. For example, in environmental protection, the goal is to identify  a source of pollution $F$ from sparse measurement of the flux, and pollution source detection plays a vital role in pollution management, including air quality forecasting and groundwater contamination studies etc \cite{GK,HJL}.

There is an extensive literature on the inverse source problem for the heat equation \eqref{eqn:heat}. These works can be grouped into two classes. One group of works focuses on distributed sources (see, e.g., \cite{Cannon:1968,HettlichRundell:2001,HuHuangJin:2025,HuangYamamoto:2020} for an incomplete list). The second group focuses the case of point sources, which has received a lot of attention in the last twenty years, due to its significant practical application in pollution detection; see, e.g., the works \cite{AndrleElBadia:2012,ElBadiaHaDuong:2002,ElBadiaHamdi:2005,HuangJinKian:2025,KomornikYamamoto:2005,LingYamamoto:2006,SunWang:2025}. These works have established the uniqueness of the inverse problem in one- and multi-dimensional cases, and have also proposed several effective numerical schemes. One common strategy for constructing the reconstruction is based on suitable regularized formulations  \cite{AndrleElBadia:2012,GuZhangZhang:2025}. There are also non-iterative methods, e.g.,  reciprocity gap functional and approximate controllability \cite{ElBadiaHaDuong:2002,QiuWangYu:2025} and range test \cite{SunWang:2025}. However, in the multi-dimensional case, most of these important works require the knowledge of flux data over the whole boundary $\partial\Omega$ or its subset. Note that in practice, the measurement is often collected by sensors, which are inherently finite. Thus it is very natural to ask whether one can reduce the amount of the boundary measurement to the observations at finitely many points on the boundary $\partial\Omega$ while still maintains the uniqueness of the recovery. The investigation of the setting is of enormous practical significance.

For a simply  connected smooth subdomain $\omega$ of the unit disc $D$ in $\mathbb{R}^2$, Hettlich and Rundell \cite[Theorem 3.1]{HettlichRundell:2001} and proved that the flux observations at two distinct points can uniquely determine the subdomain $\omega$ in the source $F(x,t)=\chi_\omega (x)$ (with $\chi_\omega$ being the characteristic function of $\omega$). The analysis relies on an explicit representation of the boundary flux in terms of the eigensystem of the negative Laplacian on the unit disc $D$. This seminal work \cite{HettlichRundell:2001} was extended and refined in several subsequent works on distributed sources, including semi-discrete in time sources and the time-fractional model etc. \cite{LiZhang:2020,LinZhangZhang:2022,RundellZhang:2020}.
Recently, Gu et al \cite{GuZhangZhang:2025} established that on the unit disc $D$, the flux data at two distinct points on the boundary $\partial D$ can uniquely determine the location $p$ of the point source of the form $g\delta(x-x_0)$, for a known constant $g$. The case of point sources is more involved since the global regularity of the solution $u$ to problem \eqref{eqn:heat} is very limited, and the derivation of the  expression of the boundary flux is more delicate. To resolve the technical challenge, Gu et al \cite{GuZhangZhang:2025} resorted to an approximate Dirac $\delta$ function and a limiting process to obtain a representation of the flux data. Note that all these works focus on the unit disc $D$ in $\mathbb{R}^2$.

In all previously reported results, the analysis has been confined to a  disc $D\in \mathbb{R}^2$, and only the uniqueness of the recovery of time-dependent sources of the form~\eqref{eqn:source}, with $g$ being piecewise constant and known, has been established. The objective of the present work is to extend the existing literature on sparse boundary measurements for the heat equation in the following four aspects: (i) the unique recovery of point sources in higher spatial dimensions, (ii) simultaneous recovery of the amplitude $g$ and the location $p$ of  sources of the form~\eqref{eqn:source}, (iii) the recovery of more general time-dependent point sources of the form~\eqref{eqn:source}, where the amplitude $g$ is not assumed necessarily piecewise constant, and (iv) analogous results for domains that are not restricted to the unit disc. The analysis is based on combining several novel arguments, including explicit representations of solutions to problem~\eqref{eqn:heat}, refined properties of the heat and Poisson kernels, and geometric considerations, notably the application of the Kellogg--Warschawski theorem. The theoretical analysis will be supplemented by numerical simulations.

In this work, we revisit the inverse problem of recovering one point source $F(x,t)$ of the form \eqref{eqn:source}, in which the amplitude $g(t)$ is allowed to be time-dependent, when compared with the work \cite{GuZhangZhang:2025}. More precisely, for the amplitude $g(t)$, we consider the following two admissible classes:
\begin{align*}
\mathcal{A}_{\rm pwc}&=\left\{g:g=\sum_{k=1}^{K}q_k\chi_{(t_{k-1},t_k]},\ t_0=0,\ t_K=T,\ q_k\mbox{ distinct},\ g(0)=q_1\neq 0.\right\},\\
\mathcal{A}_c&=\left\{g\in L^2(0,T):\supp(g)\subset [0,T_1]\mbox{ with }T_1<T,\ \int_0^{T}g\mathrm{d}t\neq 0\right\}.
\end{align*}
Note that different elements in the set $\mathcal{A}_{\rm pwc}$ can have different partitions $\{t_k\}_{k=1}^K$ and different $K$.

In this work, we obtain two new uniqueness results. The first result is about the unit ball in $\mathbb{R}^d$, $d\geq 2$, and a piecewise constant in time amplitude $g(t)$.
\begin{theorem}\label{thm:ball}
Let $\Omega\subset \RR^d$ be the unit ball with $d\geq 2$, and $g,\widetilde{g}\in \mathcal{A}_{\rm pwc}$. Let $u$ and $\widetilde{u}$ solve problem \eqref{eqn:heat} with the
 source $F=g(t)\delta(x-p)$ and $\widetilde{F}=\widetilde{g}(t)\delta(x-\widetilde{p})$, respectively. Suppose that $\partial_{\nu}u(z_{\ell},t)=\partial_\nu \widetilde{u}(z_{\ell},t)$, for any $t\in (0,T)$, at the points $z_{\ell}\in \mathbb{S}^{d-1}$, $ \ell=1,\ldots,d$. If the vectors $\{z_\ell\}_{\ell=1}^d$ are linearly independent,  then $g=\widetilde{g}$  and $p=\widetilde{p}$.
\end{theorem}

The following corollary specializes the result to the two- and three-dimensional cases. The notation $\mathbb{Z}$ and $\mathbb{Q}$ denote the sets of integers and rational numbers, respectively.
\begin{corollary}\label{cor:disc-pwc}
Let the conditions of Theorem \ref{thm:ball} hold.
\begin{itemize}
    \item[{\rm(i)}] $\Omega$ is the unit disc, and the two points $\{e^{i\theta_{\ell}}\}_{\ell=1}^2\subset\partial\Omega$ satisfy $\theta_2-\theta_1\notin \mathbb{Z}\pi$. Then $g=\widetilde{g}$ and $p=\widetilde{p}$.
    \item[{\rm(ii)}] $\Omega$ is the unit ball in $\mathbb{R}^3$, the three points $\{z_{\ell}\}_{\ell=1}^3$ and the origin are not coplanar. Then $g=\widetilde{g}$ and $p=\widetilde{p}$.
\end{itemize}
\end{corollary}

Corollary \ref{cor:disc-pwc}(i) generalizes the recent result of Gu et al \cite[Theorem 1.1]{GuZhangZhang:2025}, in which the authors proved that the location $p$ in the unit disc $D$ can be uniquely determined by the given observational data. The generalization lies in the piecewise constant temporal component, which is also uniquely determined. Moreover, the condition on the measurement points $\{z_\ell=e^{i\theta_\ell}\}_{\ell=1}^2$ is significantly relaxed from $\theta_1-\theta_2\not\in \mathbb{Q}\pi$ to $\theta_1-\theta_2\not\in\mathbb{Z}\pi$. Note that the condition $\theta_1-\theta_2\not\in \mathbb{Q}\pi$ was also imposed in the works \cite{HettlichRundell:2001,RundellZhang:2020,LiZhang:2020}. The extension is achieved by a more refined analysis of the Dirichlet series expansion of the flux data. Corollary \ref{cor:disc-pwc}(ii) is the $3$D version of (i), which appears to be new in the context of sparse boundary measurements, to the best of our knowledge.  The analysis employs the representation of the boundary flux data via eigenfunction expansion and Dirichlet heat kernel, and properties of the leading eigenfunctions on the unit ball in $\mathbb{R}^d$. In  contrast, the analysis in \cite{GuZhangZhang:2025} utilizes an approximation of the Dirac delta function, which is more involved.

The next result treats the case of a compactly supported amplitude in the two-dimensional case, possibly a general smooth bounded domain.
\begin{theorem}\label{thm:compact}
Let $\Omega\subset \mathbb{R}^2$ be a simply-connected smooth bounded domain, and $g(t),\widetilde {g}(t)\in\mathcal{A}_c$. Let $u$ and $\widetilde u$ solve problem \eqref{eqn:heat} with the source $F=g(t)\delta(x-p)$ and $\widetilde F=\widetilde g(t)  \delta(x-\widetilde p)$, respectively. Then the following two statements hold.
\begin{itemize}
\item [{\rm(i)}] Suppose  that $g=\widetilde{g}\in \mathcal{A}_c$  is known and that $\Omega$ is the unit disc. If $\partial_{\nu}u(z_{\ell},t)=\partial_\nu \widetilde{u}(z_{\ell},t)$  for all $t\in (0,T)$ and $\ell=1,2$, with $z_{\ell}=e^{i\theta_{\ell}}$ and $\theta_1-\theta_2\notin \mathbb{Z}\pi$, then $p=\widetilde{p}$.
\item [{\rm(ii)}] If $g$ and $\widetilde{g}$ are unknown, then $\partial_{\nu}u(z_{\ell},t)=\partial_\nu \widetilde{u}(z_{\ell},t),\ell=1,2,3$, for any three distinct points $\{z_{\ell}\}_{\ell=1}^3$ and $t\in (0,T)$ implies $p=\widetilde p$ and $g=\widetilde g$.
\end{itemize}
\end{theorem}

Theorem \ref{thm:compact} generalizes Corollary \ref{cor:disc-pwc}(i) in two aspects: it substantially relaxes the requirement on the time-dependent amplitude $g(t)$ from being piecewise constant to being compactly supported, and $\Omega$ is not necessary the disc, provided that one can have the observation at one more point. The proof relies heavily on a new approach based on Laplace transform, analytic continuation and properties of the Poisson kernel. The case of a general smooth domain in $\mathbb{R}^2$ is treated using Kellogg-Warschawski's theorem, a version of Riemann mapping theorem. Note that an assertion similar to Theorem \ref{thm:compact}(ii) holds also for the case of a unit all in $\mathbb{R}^d$, i.e., the flux observations at $d+1$ distinct points uniquely determine the source location $p$ and a compactly supported amplitude $g(t)\in \mathcal{A}_c$; see  Remark \ref{rmk:compact} for details.

In Section \ref{sec:numer}, we present several numerical experiments to illustrate the feasibility of reconstructing the location $p$ and amplitude $g(t)$ in the two-dimensional case. The numerical results indicate that the location $p$ can be stably recovered when the amplitude $g(t)$ is known, and the amplitude $g(t)$ can also be recovered using suitable regularization techniques.

The rest of the paper is organized as follows. In Section \ref{sec:prelim}, we give several preliminary results on the direct problem, including the representation of the boundary flux. In Section \ref{sec:proof}, we give the technical proofs of Theorems \ref{thm:ball} and \ref{thm:compact}. In Section \ref{sec:numer}, we present numerical results to illustrate the feasibility of the reconstruction. Throughout the notation $(\cdot,\cdot)$ denotes the $L^2(\Omega)$ inner product, and $C$, with or without a superscript, denotes a generic positive constant which may change at each occurrence.

\section{Preliminaries on the direct problem}\label{sec:prelim}
In this section we state several preliminary results on problem \eqref{eqn:heat} which will play a crucial role in the analysis  in Section \ref{sec:proof}.
\subsection{Regularity and extension of the solution}
In view of \cite[Theorem 1, Section 6.5.1]{Evans2010}, the negative Dirichlet Laplacian $-\Delta$ on a smooth domain $\Omega\subset \RR^d$ admits an eigen-system $\{(\lambda_n,\varphi_n)\}_{n=1}^\infty$, with the eigenvalues ordered nondecreasingly
$0<\lambda_1\leq \lambda_2\leq \cdots$ and $\lambda_n\to\infty$ as $n\to \infty$. Each eigenvalue $\lambda_n$  has a finite multiplicity.
Moreover, the eigenfunctions $\{\varphi_n\}_{n=1}^\infty\subset C^{\infty}(\overline{\Omega})$ can be taken to form an orthonormal basis of $L^2(\Omega)$.
\begin{lemma}[{\cite[Sections 4.1, 6.2.4 and 6.2.5]{Grebenkov2013}}]\label{lem:control}
The eigenpairs $\{(\lambda_n,\varphi_n)\}_{n=1}^\infty$ satisfy
\begin{align*}
\|\varphi_n\|_{L^{\infty}(\Omega)}\leq \lambda_n^{\frac{d}{4}}\quad
    \mbox{and}\quad
\|\p_{\nu}\varphi_n(x)\|_{L^{\infty}(\p\Omega)}\leq C\lambda_n^{\frac{1}{2}}.
\end{align*}
Also, the eigenvalues $\lambda_n$ satisfy the following Weyl's law:
\begin{align*}
    \lambda_n\sim Cn^{\frac{2}{d}},\quad n\to\infty.
\end{align*}
\end{lemma}

Let $K$ and $K_\Omega$ be the heat kernel on the whole space $\RR^d$ and the Dirichlet heat kernel on the domain $\Omega$, respectively, which are defined  by
\begin{equation}\label{eqn:Green-domain-flux}\begin{aligned}
    &K(x,y,t)=(4\pi t)^{-\frac{d}{2}} e^{-\frac{|x-y|^2}{4t}},\quad x,y\in\mathbb R^d,\ t>0,\\
    & K_{\Omega}(x,y,t)=\sum_{n=1}^\infty e^{-\lambda_nt}\varphi_n(x)\varphi_n(y),\quad x,y\in\Omega,\ t>0.\end{aligned}
\end{equation}
 We summarize the properties of the heat kernels in the next lemma.
\begin{lemma}\label{heat_kernel}
For the functions $K$ and $K_{\Omega}$, the following statements hold.
\begin{enumerate}
\item [{\rm(i)}] $\min (K_\Omega(x,y,t),K(x,y,t))> 0$,  $x,y\in \Omega$,  $t>0$.
\item [{\rm(ii)}]
$|K_{\Omega}(x,y,t)|\leq Ct^{-\frac{d}{2}}e^{-C'\frac{|x-y|^2}{t}}$,  $x,y\in \Omega$,  $t>0$.
\item [{\rm(iii)}] $K_{\Omega},K\in C^\infty(\overline{\Omega}\times\overline{\Omega}\times (0,+\infty))$.
\item [{\rm(iv)}] Fix  $\varepsilon>0$ and $p\in \Omega$, and define $\Omega_{\varepsilon}=\{x\in \Omega:|x-p|\geq\varepsilon\}$. Then both $K_{\Omega}(\cdot,p,\cdot)$ and $K(\cdot,p,\cdot)$ can be extended by zero to an element of $C^{\infty}(\overline{\Omega_{\varepsilon}}\times[0,+\infty))$.
    \end{enumerate}
\end{lemma}
\begin{proof}
All the statements for the free-space heat kernel $K$ can be verified directly using its explicit formula. For the kernel $K_{\Omega}$, the statements (i)--(ii) can be found at \cite[Corollary 3.2.8 and Theorem 3.3.5]{Davies1989}, and the statement (iii) follows from the fact that the series
\begin{equation*}
K_{\Omega}(x,y,t)=\sum_{n=1}^\infty e^{-\lambda_nt}\varphi_n(x)\varphi_n(y)
\end{equation*}
converges uniformly on $ \Omega\times\Omega\times [t_0,\infty)$ for any $t_0>0$ by Lemma \ref{lem:control}. We now prove (iv) for $K_{\Omega}$. Without loss of generality, we can set $\varepsilon<\mathrm{dist}(p,\p\Omega)/3$ and $B(p,\varepsilon)=\{x\in \Omega:|x-p|\leq \varepsilon\}$. Now choose a smooth  function $\zeta\in C^\infty(\overline{\Omega})$ such that $0\leq \zeta \leq 1$, $\zeta|_{B(p.\varepsilon)}=0$, $\zeta=1$ on a neighborhood of $\partial\Omega$ and $|\nabla\zeta|\leq C$. Then the function $(x,t)\mapsto \zeta(x) K(x,p,t)$ belongs to $C^{\infty}(\overline{\Omega}\times[0,\infty))$. Let
\begin{equation}\label{eqn:Kp}
K_p(x,t)=K_{\Omega}(x,p,t)-K(x,p,t)+\zeta(x) K(x,p,t),\quad (x,t)\in\Omega\times(0,+\infty).
\end{equation}
Let $T_\star>0$ and  $\phi\in C^\infty(\overline{\Omega}\times[0,T_\star])$ be such that $\phi(x,T_\star)=0$, for $x\in\overline{\Omega}$, and $\Delta^k\phi(x,t)=0$, for $(x,t)\in\partial\Omega\times(0,T_\star)$ and $k\in\mathbb N\cup\{0\}$. Then, using classical properties of the heat kernel, we deduce
\begin{align*}
\int_0^{T_\star}\int_\Omega K_\Omega(x,p,t)(-\partial_t\phi-\Delta\phi)(x,t)\d x\d t&=\phi(p,0),\\
\int_0^{T_\star}\int_\Omega K(x,p,t)(-\partial_t\phi-\Delta\phi)(x,t)\d x\d t&=\phi(p,0)-\int_0^{T_\star}\int_{\partial\Omega}K(x,p,t)\partial_\nu\phi(x,t)\d\sigma(x)\d t.
\end{align*}
Moreover, by the properties of the map $\zeta$, we obtain
\begin{align*}
&\int_0^{T_\star}\int_\Omega\zeta(x)K(x,p,t)(-\partial_t\phi-\Delta\phi)(x,t)\d x\d t\\
=&\zeta(p)\phi(p,0)-\int_0^{T_\star}\int_{\partial\Omega}\zeta(x)K(x,p,t)\partial_\nu\phi(x,t)\d\sigma(x)\d t\\
&+\int_0^{T_\star}\int_\Omega(\p_t(\zeta K)-\Delta (\zeta K))(x,p,t)\phi(x,t)\d x\d t\\
=&-\int_0^{T_\star}\int_{\partial\Omega}K(x,p,t)\partial_\nu\phi(x,t)d\sigma(x)\d t\\
&+\int_0^{T_\star}\int_\Omega(\p_t(\zeta K)-\Delta (\zeta K))(x,p,t)\phi(x,t)\d x\d t.
\end{align*}
By combining these identities with \eqref{eqn:Kp}, we find
\begin{equation}\label{idd}\int_0^{T_\star}\int_\Omega K_p(x,t)(-\partial_t\phi-\Delta\phi)(x,t)\d x\d t=\int_0^{T_\star}\int_\Omega(\p_t(\zeta K)-\Delta (\zeta K))(x,p,t)\phi(x,t)\d x\d t.\end{equation}
In view of (iii) and the definition of $\zeta$, we have
$$(\p_t(\zeta K)-\Delta (\zeta K))(\cdot,p,\cdot)=-2\nabla\zeta\cdot\nabla K(\cdot,p,\cdot)-(\Delta\zeta)K(\cdot,p,\cdot)\in C^\infty(\overline{\Omega}\times[0,T_\star]).$$
Thus, by density, we can extend the identity \eqref{idd} to any $\phi\in H^1(0,T_\star;L^2(\Omega))\cap L^2(0,T_\star:H^2(\Omega)\cap H^1_0(\Omega))$ such
that $\phi(x,T_1)=0$, for $x\in\Omega$. Hence, $K_p$ is the unique solution in the transposition sense of the problem
\begin{equation*}
   \left\{ \begin{aligned}
    \partial_tK_p-\Delta K_p  &=(\p_t(\zeta K)-\Delta (\zeta K))(\cdot,p,\cdot),\quad \mbox{in } \Omega\times (0,T_\star),\\
    K_p&=0
    ,\quad\mbox{on }\partial\Omega\times (0,T_\star),\\
        K_p&=0,\quad \mbox{in }\Omega\times \{0\}.
    \end{aligned}\right.
\end{equation*}
Let $f_p(x,t)=(\p_t(\zeta K)-\Delta_x (\zeta K))(x,p,t)$, $(x,t)\in \overline{\Omega}\times[0,T_\star]$. Then, $f_p\in C^\infty(\overline{\Omega}\times[0,T_\star])$ and, for any $k\in \mathbb{N}$,
\[\p_t^kf_p(x,t)=-2\nabla\zeta(x)\cdot \nabla \p_t^kK(x,p,t)-\p_t^kK(x,p,t)\Delta \zeta(x),\quad (x,t)\in \overline{\Omega}\times(0,T_\star).\]
Since $K(x,y,t)$ decays exponentially near $t=0$ when $|x-y|>\varepsilon$, we have $\p_t^kf_p(x,t)|_{t=0}=0$ for any $k\in \mathbb{N}$ and $x\in \overline{\Omega}$. The standard parabolic regularity theory \cite[Chapter 7, Theorem 7]{Evans2010} implies that $K_p\in C^{\infty}( \overline{\Omega}\times [0,T_\star])$ for any $T_\star>0$. Thus,
$K_p\in C^{\infty}( \overline{\Omega}\times [0,\infty)).$ Now it follows that
\[K_{\Omega}(\cdot,p,\cdot)|_{\overline{\Omega_{\varepsilon}} \times[0,\infty)}=K_p|_{\overline{\Omega_{\varepsilon}} \times[0,\infty)}+(1-\zeta)K(\cdot,p,\cdot)|_{\overline{\Omega_{\varepsilon}} \times[0,\infty)}\in C^{\infty}( \overline{\Omega_{\varepsilon}} \times[0,\infty)).
\]
This completes the proof of the lemma.
\end{proof}

Let $H(t)$ be the Heaviside function and $K_p$ be given by \eqref{eqn:Kp}. By induction on $k\in\mathbb N\cup\{0\}$, we can verify $\p_t^kK_p(x,t)|_{t=0}=0$ for $x\in\overline{\Omega_{\varepsilon}}$. 
Recall that, the solution $u$ in the transposition sense of \eqref{eqn:heat} can be represented by
\[u(x,t)=\int_0^tg(t-s)K_{\Omega}(x,p,s)\mathrm{d}s,\quad (x,t)\in\Omega\times(0,T),\]
and, we have
\[u(x,t)=(Hg)*_{(t)}(HK_p(\cdot,x)),\quad (x,t)\in\overline{\Omega_{\varepsilon}}\times(0,T),\]
which implies that $u|_{\overline{\Omega_{\varepsilon}}\times[0,T]}$ can be extended by $0$ to an element of
$ C^{\infty}( \overline{\Omega_{\varepsilon}}\times[0,T])$. Using these properties, we can extend $u$   to a smooth function, still denoted by $u$, defined on $\overline{\Omega_{\varepsilon}}\times[0,+\infty)$ by
\[u(x,t)=\int_0^tg(t-s)K_{\Omega}(x,p,s)\mathrm{d}s,\quad (x,t)\in\overline{\Omega_{\varepsilon}}\times[0,+\infty).\]
Here and below, the map $g$  denotes the extension by zero of $g$ to $(0,+\infty)$. This extension $u$ satisfies
$\p_{\nu}u\in C^{\infty}(\partial\Omega\times [0,\infty))$.
Since $\p_{\nu}K_{\Omega}\in C^{\infty}([0,T]\times \p\Omega)$ and $g\in L^1(0,T)$ for either $g\in \mathcal{A}_{\rm pwc}$ or $g\in \mathcal{A}_c$, the normal derivative $\partial_\nu u$ of $u$ is given by
\begin{equation}\label{eqn:flux-rep}
\p_{\nu}u(z,t)=\int_{0}^tg(t-s)\p_{\nu}K_{\Omega}(z,p,s)\mathrm{d}s,\quad (z,t)\in\partial\Omega\times[0,+\infty).
\end{equation}
Now we state an analytic  extension result.
\begin{lemma}\label{lem:ana-ext}
Suppose $g\in \mathcal{A}_c$ and $\supp (g)\subset [0,T_1]$. Then for any fixed $x\in\partial\Omega$, the map $t\mapsto \p_{\nu}u(x,t)$ is analytic on $(T_1,+\infty)$.
\end{lemma}
\begin{proof}
Like before, choose a smooth cut-off function $\zeta(x)$ such that $0\leq \zeta \leq 1$, $\zeta|_{B(p.\varepsilon)}=0$, $\zeta|_{\Omega\backslash B(p,2\varepsilon)}=1$ and $|\nabla\zeta|\leq C$.
Then $K_{\Omega}(x,p,t)=K_p(x,t)$ for $x\in \overline{\Omega}\backslash B(p,2\varepsilon)$, cf. \eqref{eqn:Kp}, and hence the normal derivative $\partial_\nu u(z,t)$ can be expressed as
\begin{align*}
\p_{\nu}u(z,t)=\int_0^{T_1}g(s)\p_{\nu}K_{\Omega}(z,p,t-s)\mathrm{d}s,\quad  t\in(T_1,+\infty),\ z\in\partial\Omega.
\end{align*}
Clearly, we have
\[
\partial_{\nu}K_{\Omega}(z,p,t-s)=\sum_{n=1}^\infty e^{-\lambda_n(t-s)}\partial_{\nu}\varphi_n(z)\varphi_n(p)\quad z\in\partial\Omega,\ t\in(T_1,+\infty),\ s\in(0,T_1).
\]
By Lemma \ref{lem:control} and dominated convergence theorem, we obtain
\[
\partial_{\nu}u(z,t)=\sum_{n=1}^\infty \left(\int_0^{T_1}g(s)e^{-\lambda_n (t-s)}\mathrm{d}s\right)\p_{\nu}\varphi_n(z)\varphi_n(p).
\]
Now fix the set of complex number $\mathbb C_{T_1}:=\{\tau\in\mathbb C:\ \Re(\tau)>T_1\}$ and define  the map $U$  by
$$U(z,\tau):=\sum_{n=1}^\infty \left(\int_0^{T_1}g(s)e^{-\lambda_n (\tau-s)}\mathrm{d}s\right)\p_{\nu}\varphi_n(z)\varphi_n(p),\quad (z,\tau)\in \partial \Omega\times \mathbb C_{T_1}.$$
In light of Lemma \ref{lem:control}, one can easily check that for any $z\in\partial\Omega$ and any compact set $\Sigma$ of $\mathbb C$, $\Sigma\subset \mathbb C_{T_1}$, the sequence
$$\sum_{n=1}^N \left(\int_0^{T_1}g(s)e^{-\lambda_n (\tau-s)}\mathrm{d}s\right)\p_{\nu}\varphi_n(z)\varphi_n(p),\quad N\in\mathbb N$$
converges uniformly with respect to $\tau\in \Sigma$. This proves that $\tau \mapsto U(z,\tau)$ is holomorphic on $\mathbb C_{T_1}$. By combining this with the fact that $U(z,t)=\partial_{\nu}u(z,t)$, $t\in(T_1,+\infty)$, we obtain the desired result.
\end{proof}

\subsection{Heat kernel, Green's function and Poisson kernel}
For any $x\in \Omega$, there exists Green's function $G(x,y)$ satisfying
\begin{align*}
\left\{\begin{aligned}
    -\Delta_y G(x,y) &=\delta(x-y),\quad y\mbox{ in } \Omega,\\
    G(x,y)&=0,\quad y\mbox{ on }\p\Omega.
\end{aligned}\right.
\end{align*}
The function $G(x,y)$ is symmetric in $x$ and $y$, and is smooth for any $x\in \overline{\Omega},x\neq y$. Thus for any $x\in\Omega$ and $z\in\p\Omega$, we can define the Poisson kernel $P(x,z)$ to be
\[
P(x,z)=\p_{\nu}G(x,z)=\nu(z)\cdot \nabla_z G(x,z),
\]
where $\nu(z)$ is the unit outward normal vector to the boundary $\partial\Omega$ at the point $z\in \partial\Omega$.
By definition, for any harmonic function $f\in C^2(\overline{\Omega})$ and $x\in \Omega$, we have
$f(x)=\int_{\p\Omega}P(x,z)f(z)\mathrm{d}\sigma_z$.

The following important relation holds.
\begin{lemma}\label{Green_heat}
For any $x\neq y$ and $\lambda\geq 0$, define $G_\lambda(x,y)$ by
$G_{\lambda}(x,y)=\int_0^{\infty}e^{-\lambda t}K_{\Omega}(x,y,t)\mathrm{d}t.$ Then $G_{\lambda}$ is smooth when $|x-y|\geq \varepsilon$ for any $\varepsilon>0$, and solves
\begin{equation}\label{eqn:G-lam}
\left\{\begin{aligned}
    (\lambda I-\Delta_y) G_{\lambda}(x,y) &=\delta(x-y),\quad y\mbox{ in } \Omega,\\
    G_{\lambda}(x,y)&=0,\quad y\mbox{ on }\p\Omega.
\end{aligned}\right.
\end{equation}
Consequently, $\int_0^{\infty}\p_{\nu}K_{\Omega}(x,z,t)\mathrm{d}t=P(x,z)$ for $z\in \p\Omega$.
\end{lemma}
\begin{proof}
Let $A=\Delta$ be the Dirichlet Laplacian on the domain $\Omega$. For $s\geq0$, we define the Hilbert space  $X^s$ by
\[
X^s=\bigg\{u\in L^2(\Omega):u=\sum_{n=1}^{\infty}(u,\varphi_n)\varphi_n,\ \sum_{n=1}^{\infty}\lambda_n^s|(u,\varphi_n)|^2<+\infty\bigg\},
\]
equipped with the norm $\|u\|_s^{2}=\sum_{n=1}^{\infty}\lambda_n^s|(u,\varphi_n)|^2$, and let $X^{-s}$ be the dual space of $X^s$. Then by Sobolev embedding, we have $\delta(x-y)\in X^{-s}$ for any $s>\frac{d}{2}$. We can define $e^{tA}$ for any $t\geq 0$ as a bounded operator from $X^{-s}$ to itself by
\[
e^{tA}u=\sum_{n=1}^\infty e^{-\lambda_nt}c_n\varphi_n,\quad u=\sum_{n= 1}^\infty c_n\varphi_n\in X^{-s},
\]
with $\|e^{tA}u\|_{-s}\leq e^{-\lambda_1t}\|u\|_{-s}$.
Now we claim that $e^{tA}$ is a $C_0$ semigroup on $X^{-s}$. In fact, by definition, it is easy to verify $e^{(t+s)A}=e^{tA}e^{sA}$.
When $t\to 0$, for any $u=\sum_{n=1}^{\infty}c_n\varphi_n$, we have
\[\|e^{tA}u-u\|_{-s}=\sum_{n= 1}^{\infty}\lambda_n^{-s}(1-e^{-\lambda_nt})^2c_n^2\to 0,
\]
by the estimate
\[
\sum_{n= 1}^{\infty}\lambda_n^{-s}(1-e^{-\lambda_nt})^2c_n^2\leq 4\sum_{n= 1}^{\infty}\lambda_n^{-s}c_n^2\leq 4\|u\|_{-s}\]
and the dominated convergence theorem. Note also that
\[
\lim_{t\to0}\frac{e^{tA}u-u}{t}=\lim_{t\to 0}\sum_{n=1}^\infty\frac{e^{-\lambda_n t}-1}{t}c_n\varphi_n=-\sum_{n=1}^\infty\lambda_n c_n\varphi_n=Au,\quad \forall u=\sum_{n= 1}^\infty c_n\varphi_n\in X^{2-s}.
\]
This implies $X^{2-s}\subset D(L)$, with $L$ being the generator of $\{e^{tA}\}_{t\geq0}$. Further, if $v=\sum_{n= 1}^\infty b_n\varphi_n\in D(L)$ is defined by
\[
v=\lim_{t\to 0}\frac{e^{tA}u-u}{t},\quad \mbox{for } u=\sum_{n= 1}^\infty c_n\varphi_n\in X^{-s},
\]
then we have
\[\left|b_n-\frac{e^{-t\lambda_n}-1}{t}c_n\right|^2\leq \left\|\frac{e^{-tA}u-u}{t}-v\right\|_{-s}\to 0\quad \mbox{as }t\to 0.\]
Thus $b_n=-\lambda_nc_n$ and $D(L)=X^{2-s}$.
In sum, $\{e^{tA}\}_{t\geq 0}$ defines a $C_0$ semigroup on $X^{-s}$ for any $s\geq 0$ with generator $(A,X^{2-s})$. By \cite[Chapter 1, Remark 5.4]{Pazy1983}, we deduce
\begin{equation*}
\int_0^\infty e^{-\lambda t}e^{tA}u_0\mathrm{d}t=(\lambda I-A)^{-1}u_0,\quad
\forall u_0\in X^{-s},  \forall \lambda> -\lambda_1.
\end{equation*}
By taking $u_0=\delta(x-y)$, we get the first statement. Note that $\p_{\nu}K_\Omega(x,z,t)\in C^{\infty}([0,T)\times \p\Omega)$. The second statement now follows by taking $\lambda=0$ and the definition of the Poisson kernel $P(x,z)$.
\end{proof}

\subsection{Dirichlet eigenfunctions on the unit ball}
In this part, we describe basic properties about the Dirichlet eigenfunctions on the unit ball $\Omega\subset \RR^d$, which will play a fundamental role in the proof of Theorem \ref{thm:ball}. The eigenfunctions are closely related to the Bessel function of the first kind. Recall that the Bessel function $J_\beta(x)$ of the first kind of order $\beta\geq 0$ is defined by
\[J_{\beta}(x)=\sum_{k=0}^{\infty}\frac{(-1)^k}{k!\Gamma(\beta+k+1)}\left(\frac{x}{2}\right)^{2k+\beta},\quad x\geq 0.\]
The function $J_{\beta}(x)$ is uniquely determined to be the solution of the ODE
\[x^2y''+xy'+(x^2-\beta^2)y=0,\quad x\geq 0\]
that is non-singular at the origin $x=0$.

The next lemma summarizes basic properties of $J_\beta(x)$.
\begin{lemma}[{\cite{Harry1967,Olver2010}}]\label{lem:Bessel}
Let $J_{\beta}(x)$ to be the Bessel function of the first kind of order $\beta\geq 0$. Then the following properties hold.
\begin{itemize}
\item [{\rm(i)}] Every $J_{\beta}(x)$ has infinitely many real roots and no complex roots. Every positive root is simple. We can list the roots increasingly and denote the $l$th positive root of $J_{\beta}(x)$ to be $\{s_{\beta,l}\}_{l=1}^\infty$. The negative roots are just $\{-s_{\beta,l}\}_{l=1}^\infty$.
        \item  [{\rm(ii)}] The following asymptotics  of $J_{\beta}(x)$ hold:
    \begin{align*} J_{\beta}(x)&\sim \frac{1}{\Gamma(\beta +1)}\left(\frac{x}{2}\right)^{\beta}\quad \mbox{as }x\to 0,\\
J_{\beta}(x) &\sim \sqrt{\frac{\pi }{2x}}\cos\left(x-\frac{\beta}{2} \pi-\frac{\pi}{4}\right)+o(1)\quad \mbox{as }x\to\infty\mbox{ is real}.
\end{align*}
\item [{\rm(iii)}] For any fixed $\beta\geq 0$, the following asymptotic formula holds
\[
s_{\beta,l}=\Big(l+\frac{\beta}{2}-\frac{1}{4}\Big)\pi+O(l^{-1})\quad \mbox{when }l\to\infty.
\]
\item [{\rm(iv)}] For any fixed $\beta\geq0$, let $\{s_{\beta,l}\}_{l=1}^\infty$ be the positive roots of $J_{\beta}(x)$. Then $\{\sqrt{x}J_{\beta}(s_{\beta,l}x)\}_{l=1}^{\infty}$ forms an orthogonal basis of $L^2(0,1)$. Any $f\in L^2(0,1)$ can be expressed as
\begin{align*} f(x)=\sum_{l= 1}^\infty c_{\beta l}\sqrt{x}J_{\beta}(s_{\beta,l}x)\quad\mbox{in }L^2(0,1),\quad \mbox{with }
c_{\beta l} =\frac{\int_0^1f(x)\sqrt{x}J_{\beta}(s_{\beta,l}x)\mathrm{d}x}{\int_0^1xJ^2_{\beta}(s_{\beta,l}x)\mathrm{d}x},
\end{align*}
and
\[\|f\|_{L^2(0,1)}^2=\sum_{l= 1}^{\infty}\left(\int_0^1f(x)\sqrt{x}J_{\beta}(s_{\beta,l}x)\mathrm{d}x\right)^2.\] Moreover, the convergence is pointwisely when $0<x<1$ and $f\in C^{\infty}([0,1])$.
\end{itemize}
\end{lemma}

Now we can give the following description of Dirichlet eigenfunctions on the unit ball.
\begin{theorem}\label{thm:Dirichlet_eigenfunc}
Let $\Omega$ be the unit ball in $\mathbb{R}^d$. The $L^2(\Omega)$ normalized eigenfunctions $u_{klm}$ for the Dirichlet Laplacian can be chosen to be
    \[
    u_{klm}(r,\theta)=\omega_{kl}j_k(\alpha_{kl}r)Y_{k}^m(\theta),\quad (r,\theta)\in [0,1]\times \mathbb{S}^{d-1},
    \]
    with
    $$
    j_k(\alpha_{kl}r)=r^{1-\frac{d}{2}}J_{k+\frac{d}{2}-1}(\alpha_{kl}r)\quad \mbox{and} \quad \omega_{kl}^{-2}=\tfrac12(J_{k+\frac{d}{2}}(\alpha_{kl}))^2,$$
    and $Y_k^m$ are normalized spherical harmonics, where $k,l,m\in\mathbb{N},$ $l\geq 1$ and $1\leq m\leq d_k$ with $d_k$ given by
\[d_0=1,\quad d_1=d,\quad d_{k}=\binom{k+d-1}{k}-\binom{k+d-2}{k-1}.\]
The numbers $\{\alpha_{kl}=s_{k+\frac{d}{2}-1,l}\}_{l=1}^\infty$ are positive roots of the Bessel function $J_{k+\frac{d}{2}-1}(x)$ that are ordered increasingly.
\end{theorem}
\begin{proof}
    Using the polar coordinates $(r,\xi)\in [0,\infty)\times \mathbb{S}^{d-1}$, the Laplacian of a function $f(x)=F(r)G(\xi)$ can be written as  \cite[p. 34]{Chavel1984}
    \[\Delta_{\RR^d} f=r^{1-d}\p_{r}(r^{d-1}\p_r F)G+r^{-2}F\Delta_{\mathbb{S}^{d-1}}G,\]
    where $\Delta_{\mathbb{S}^{d-1}}$ denotes the spherical Laplacian.
    Thus by means of separation of variables, we deduce that an eigenfunction $\varphi(x)=R(r)G(\xi)$ corresponding to the eigenvalue $\lambda>0$ should satisfy
    \[\Delta_{\mathbb{S}^{d-1}}G=-\mu G\quad
    \mbox{and}\quad R''+\frac{d-1}{r}R'+\left(\lambda-\frac{\mu}{r^2}\right)R=0\quad\mbox{with } R(1)=0,\]
    for some $\mu\in \RR$.
    That is, the function $G\colon \mathbb{S}^{d-1}\to \RR$ is an eigenfunction of $-\Delta_{\mathbb{S}^{d-1}}$. It is known that (see \cite[pp. 34-35]{Chavel1984} and \cite{Stein1971}) the space of harmonic polynomials which are homogeneous of degree $k$, when restricted to the sphere $\mathbb{S}^{d-1}$, constitutes the eigenspace of the $k$th distinct eigenvalue $\mu_k=k(k+d-2)$ of the spherical Laplacian $-\Delta_{\mathbb{S}^{d-1}}$. These polynomials, when restricted to $\mathbb{S}^{d-1}$ and normalized in $L^2(\mathbb{S}^{d-1})$, are precisely spherical harmonics. The multiplicity of $\mu_k$ is exactly $d_k$. The spherical harmonics form an orthonormal basis of $L^2(\mathbb{S}^{d-1})$.

    Below we denote by $\{Y_k^m\}_{m=1}^{d_k}$ an orthonormal basis of the eigenspace corresponding to $\mu_k$. By substituting $\mu_k=k(k+d-2)$, the governing equation of $R$ reads
    \[
    R''+\frac{d-1}{r}R'+\left(\lambda-\frac{k(k+d-2)}{r^2}\right)R=0,\quad \mbox{with } R(1)=0.
    \]
   Let $h(r)=r^{\frac{d}{2}-1}R(r)$. Then the function $h(r)$ satisfies
    \[h''+\frac{1}{r}h'+\left(\lambda-\frac{(k+\frac{d}{2}-1)^2}{r^2}\right)h=0,\quad\mbox{with } h(1)=0.\]
    By substituting $y=\sqrt{\lambda}r$, the function $H(y)=h(\frac{y}{\sqrt{\lambda}})$, $0\leq y\leq \sqrt{\lambda}$ satisfies the Bessel equation
    \[y^2\frac{\mathrm{d}^2H}{\mathrm{d}y^2}+y\frac{\mathrm{d}H}{\mathrm{d}y}+\left(y^2-\Big(k+\frac{d}{2}-1\Big)^2\right)H=0,\quad\mbox{with } H(\sqrt{\lambda})=0.\]
Note that all the eigenfunctions should be smooth in $\Omega$, which enforces
\[
h(r)=H(\sqrt{\lambda}r)=J_{k+\frac{d}{2}-1}(\sqrt{\lambda }r)
\]
to be the Bessel function of the first kind of order $k+\frac{d}{2}-1$. The boundary condition $H(\sqrt{\lambda})=0$ implies  that $\sqrt{\lambda}$ should be a positive root of $J_{k+\frac{d}{2}-1}(x)$, i.e., $\lambda=s^2_{k+\frac{d}{2}-1,l}$ for some $l$.

Now we claim that the sequence $\{u_{klm}\}$ forms an orthogonal basis of $L^2(\Omega)$, which will complete the proof of the lemma. Indeed, if the claim is true, then every eigenfunction $\psi$ corresponding to the eigenvalue $\lambda$ can be expanded into
\[
\psi(r,\theta)=\sum_{k,l=1}^\infty \sum_{m=1}^{d_k}c_{klm}u_{klm}(r,\theta).
\]
Note that the eigenfunction $\psi\in C^{\infty}(\overline{\Omega})$ (since the unit ball $\Omega$ is smooth). By considering
\[
-\Delta\psi=\sum_{k,l=1}^\infty \sum_{m=1}^{ d_k}\alpha^2_{kl}c_{klm}u_{klm}=\lambda\sum_{k,l=1}^\infty \sum_{m=1}^{d_k}c_{klm}u_{klm}
\] and by comparing the coefficients, we deduce $\lambda=\alpha^2_{kl}$ for some $k,l\geq 1$. Thus $\psi$ is a linear combination of finitely many $u_{klm}$ that corresponds to the eigenvalue $\alpha^2_{kl}$. Now we prove the claim. Note that
    \[\int_{\Omega}u_{klm}u_{k'l'm'}\mathrm{d}x=\omega_{kl}\omega_{k'l'}\int_0^1j_{k}(\alpha_{kl}r)j_{k'}(\alpha_{k'l'}r)r^{d-1}\mathrm{d}r\int_{\mathbb{S}^{d-1}}Y_{k}^m(\theta)Y_{k'}^{m'}(\theta)\mathrm{d}\theta.
\]
The orthogonality of spherical harmonics $Y_{k}^m$ in $L^2(\mathbb{S}^{d-1})$ implies that the integral is zero unless $k=k'$ and $m=m'$. Hence the above integral reduces to
\begin{align*}
\int_{\Omega}u_{klm}u_{k'l'm'}\mathrm{d}x&=\omega_{kl}\omega_{kl'}\int_0^1j_{k}(\alpha_{kl}r)j_{k}(\alpha_{kl'}r)r^{d-1}\mathrm{d}r\\
&=\omega_{kl}\omega_{kl'}\int_0^1J_{k+\frac{d}{2}-1}(\alpha_{kl} r)J_{k+\frac{d}{2}-1}(\alpha_{kl'} r)r\d r.
\end{align*}
By Lemma \ref{lem:Bessel} (iv), we deduce that the integral is zero unless $l=l'$. Thus the sequence  $\{u_{klm}\}$ is an orthogonal system.
For any $f\in L^2(\Omega)$, we will prove Parseval's identity, which is equivalent to the completeness. For any fixed $r$, we can write
\begin{align}\label{eqn:ckm} f(r,\theta)=\sum_{k=1}^\infty \sum_{m=1}^{d_k}c_{km}(r)Y_k^m(\theta) \quad \mbox{ in }
L^2(\mathbb{S}^{d-1}),\quad \mbox{with }
c_{km}(r)=\int_{\mathbb{S}^{d-1}}f(r,\theta)Y_k^m(\theta)\mathrm{d}\theta.
\end{align}
    Thus
\[ \int_{\mathbb{S}^{d-1}}f^2(r,\theta)\mathrm{d}\theta=\sum_{k=1}^\infty \sum_{m=1}^{ d_k}c^2_{km}(r),
\]
    and
    \[\int_{\Omega}f^{2}(x)\mathrm{d}x=\int_0^1\int_{\mathbb{S}^{d-1}} r^{d-1}f^2(r,\theta)\mathrm{d}\theta\mathrm{d}r=\sum_{k=1}^\infty\sum_{m=1}^{ d_k}\int_{0}^1r^{d-1}c^2_{km}(r)\mathrm{d}r.\]
    Since $\|c_{km}(r)r^{\frac{d}{2}-\frac{1}{2}}\|_{L^2(0,1)}\leq \|f\|_{L^2(\Omega)}$, by Lemma \ref{lem:Bessel} (iv), we obtain
    \[\int_{0}^1r^{d-1}c^2_{km}(r)\mathrm{d}r=\sum_{l=1}^{\infty}\omega_{kl}^2\left(\int_0^1r^{\frac{d}{2}}c_{km}(r)J_{k+\frac{d}{2}-1}(\alpha_{kl}r)\mathrm{d}r\right)^2.\]
By substituting the expression \eqref{eqn:ckm} of $c_{km}$, we obtain
    \[\int_{0}^1r^{d-1}c^2_{km}(r)\mathrm{d}r=\sum_{l= 1}^\infty \left(\int_{\Omega}f(r,\theta)u_{klm}(r,\theta)\mathrm{d}x\right)^2.\]
    In summary, there holds
    \[\int_{\Omega}f^{2}(x)\mathrm{d}x=\sum_{k=1}^\infty \sum_{m=1}^{ d_k}\int_{0}^1r^{d-1}c^2_{km}(r)\mathrm{d}r=\sum_{k,l=1}^\infty \sum_{m=1}^{ d_k}\left(\int_{\Omega}f(r,\theta)u_{klm}(r,\theta)\mathrm{d}x\right)^2.\]
This completes the proof of the theorem.
\end{proof}
\section{Proof of the main results}
\label{sec:proof}
In this section we prove Theorems \ref{thm:ball} and \ref{thm:compact}.

\subsection{Proof of Theorem \ref{thm:ball}}

Now we can state the proof of Theorem \ref{thm:ball}.
\begin{proof}
Let the partitions of $g$ and $\widetilde g$ be given by $I_g=\{t_1,\cdots,t_K\}$ and $I_{\widetilde{g}}=\{\widetilde{t}_1,\cdots,\widetilde{t}_{\widetilde{K}}\}$, respectively. We break the lengthy proof into three steps.

\noindent\textbf{Step 1}. At this step we derive two key identities \eqref{eqn:id-zero}--\eqref{eqn:id-first} using the eigenexpansion. Let  $\rho=\min\{t_1,\widetilde{t}_1\}$. Suppose $g=q_1$ and $\widetilde{g}=\widetilde{q}_1$ with $0\neq q_1,\widetilde{q}_1\in \mathbb{R}$, when $0<t<\rho$. Then if $t<\rho$, the identity $\p_\nu u(z_{\ell},t)=\p_\nu \widetilde{u}(z_{\ell},t)$ for $\ell=1,\ldots,d$ and the representation \eqref{eqn:flux-rep} imply
\[
q_1\int_0^t\p_{\nu}K_{\Omega}(z_{\ell},p,s)\mathrm{d}s=\widetilde{q}_1\int_0^t\p_{\nu}K_{\Omega}(z_{\ell},\widetilde{p},s)\mathrm{d}s,\quad \ell=1,\ldots,d.
\]
By taking the derivative in $t$, we get
\[
q_1\p_{\nu}K_{\Omega}(z_{\ell},p,t)=\widetilde{q}_1\p_{\nu}K_{\Omega}(z_{\ell},\widetilde{p},t),\quad \ell=1,\ldots,d.
\]
Then using the series representation \eqref{eqn:Green-domain-flux} of the heat kernel $K_\Omega(t,z,p)$, we have
\[\sum_{n=1}^\infty e^{-\lambda_nt}(q_1\varphi_n(p)-\widetilde{q}_1\varphi_n(\widetilde{p}))\p_{\nu}\varphi_n(z_{\ell})=0,\quad \ell=1,\ldots,d.
\]
Now the uniqueness of the generalized Dirichlet series implies
\begin{equation}\label{eq2}
\sum_{\lambda_n=\lambda_j}\p_{\nu}\varphi_n(z_{\ell})(q_1\varphi_n(p)-\widetilde{q}_1\varphi_n(\widetilde{p}))=0.
\end{equation}
By Theorem \ref{thm:Dirichlet_eigenfunc}, all the eigenfunctions $\varphi_n$ are of the form $\omega_{kl}j_k(\alpha_{kl}r)Y_k^m(\theta)$ for some integers $k$, $l$ and $m$, and the corresponding eigenvalue $\lambda_n=\alpha^2_{kl}$ has a multiplicity $d_k$. In particular, when $k=0$, the multiplicity $d_0$ is one  and when $k=1$, the multiplicity $d_1$ is $d$. Also, note that
    \[\p_{\nu}u_{klm}(z)=\p_ru_{klm}(r,\theta)|_{r=1}=\omega_{kl}\p_rj_k(\alpha_{kl})Y_k^{m}(\theta),\quad z=(1,\theta)\in \p\Omega.\]
Now consider the eigenspace $V_1$ corresponding to the first nonzero eigenvalue of the spherical Laplacian $-\Delta_{\mathbb{S}^{d-1}}$. We can verify that the coordinate projections $p_{m}(x)=x_m,x=(x_1,\cdots,x_d)\in \mathbb{S}^{d-1}$ for $m=1,\ldots, d$ form a basis of the space $V_1$ by the definition of spherical harmonics. Since $\{Y_1^m(\theta)\}_{m=1}^d\subset V_1$ are linearly independent, there exists an invertible matrix $M\in \mathbb{R}^{d\times d}$ such that \[[Y_1^1(x),\cdots,Y_1^d(x)]^{T}=M[p_{1}(x),\cdots,p_d(x)]^{T},\quad x\in \mathbb{S}^{d-1}.\]
When the vectors $\{z_{\ell}\}_{\ell=1}^d\subset \mathbb{S}^{d-1}$ are linearly independent, the matrix $A=[p_m(z_{\ell})]_{d\times d}$ is invertible, since the $\ell${th} row is just $z_{\ell}$. Thus for some $C_l \neq0$,
    \[\det(\p_{\nu}u_{1lm}(z_{\ell}))=C_l\det M\cdot \det A\neq 0.\]
By choosing $\lambda_j=\alpha^2_{0l}$ and $\lambda_j=\alpha^2_{1l}$ in the identity \eqref{eq2} for any $l\in \mathbb{N}$, $l\geq 1$, since the coefficient matrix $[\p_{\nu}\varphi_n(z_{\ell})]_{d\times d}=[\p_{\nu}u_{1lm}(z_{\ell})]_{d\times d}$ is invertible, we get
\begin{align}
q_1j_0(\alpha_{0l}r)&=\widetilde{q}_1j_0(\alpha_{0l}\widetilde{r}), \quad l\in\mathbb{N},\label{eqn:id-zero}\\
q_1j_1(\alpha_{1l}r)Y_1^m(\theta_p)&=\widetilde{q}_1j_1(\alpha_{1l}\widetilde{r})Y_1^m(\theta_{\widetilde{p}}),\quad l\in \mathbb{N}, m=1,\ldots,d,\label{eqn:id-first}
\end{align}
where $(r,\theta_p)$ and $(\widetilde{r},\theta_{\widetilde{p}})$ are polar coordinates for $p$ and $\widetilde{p}$, respectively.

\noindent \textbf{Step 2}. Now we prove $p=\widetilde{p}$ and $q_1=\widetilde{q}_1$. When $r,\widetilde{r}\neq 0$, we can choose a smooth function $f\colon[0,1]\to\RR$ with compact support, and an open interval $(a,b)\subset [0,1]$ such that  $f(r)=1,f(\widetilde{r})=0$ and $a<r,\widetilde{r}<b$. By Lemma \ref{lem:Bessel} (iv), the Fourier-Bessel expansion (with $\alpha_{0l}$ being roots of $J_{\frac{d}{2}-1}(x)$)
\[
x^{\frac{d}{2}-1}f(x)=\sum_{l=1}^\infty c_lJ_{\frac{d}{2}-1}(\alpha_{0l}x)
\]
converges pointwisely when $x\in(a,b)$. This implies
\[
0\neq q_1f(r)=q_1\sum_{l=1}^\infty c_lr^{1-\frac{d}{2}}J_{\frac{d}{2}-1}(\alpha_{0l}r)=q_1\sum_{l=1}^\infty c_lj_0(\alpha_{0l}r)=\widetilde{q}_1\sum_{l=1}^\infty c_lj_0(\alpha_{0l}\widetilde{r})=\widetilde{q}_1f(\widetilde{r})=0,
\]
which is however impossible, since $q_1\neq 0$ and $f(r)=1$. Thus  we deduce $r=\widetilde{r}$. Note that $j_0(\alpha_{01}r)$ and $j_1(\alpha_{11}r)$ cannot vanish since $0<r,\widetilde{r}<1$ and $\alpha_{01}$ and $\alpha_{11}$ are the smallest positive roots for $j_0$ and $j_1$, respectively. This and \eqref{eqn:id-zero} imply $q_1=\widetilde{q}_1$ and $Y_1^m(\theta)=Y_1^m(\widetilde{\theta})$. Using the identity
\[
\theta=[p_1(\theta),\cdots,p_{d}(\theta)]^{T}=M^{-1}[Y_1^1(\theta),\cdots,Y_1^{d}(\theta)]^{T},
\]
we obtain $\theta_p=\theta_{\widetilde{p}}$ and thus $p=\widetilde{p}$.  When one of $r,\widetilde{r}$ is zero, say $\widetilde{r}=0$,
    then we find $J_{\frac{d}{2}-1}(\alpha_{0l}r)=c$ is a constant which is nonzero (Lemma \ref{lem:Bessel} (ii)). This is again impossible when $r\neq 0$ since $|J_{\frac{d}{2}-1}(\alpha_{0l}r)|\to 0$ when $l\to\infty$, cf. Lemma \ref{lem:Bessel}(ii).
    In summary, we have $r=\widetilde{r}=0$. Now $q_1j_0(0)=\widetilde{q}_1j_0(0)$ implies $q_1=\widetilde{q}_1$, and thus we also obtain $p=\widetilde{p}=0$ and $q_1=\widetilde{q}_1$.

\noindent\textbf{Step 3}. At this step, we prove $g\equiv \widetilde g$. Since both $g$ and $\widetilde{g}$ are piecewise constant, $h\equiv g-\widetilde{g}$ is also piecewise constant. Let
\[
t^*=\sup\{t:h(s)=0\text{~for~}0<s<t\}.
\]
Note that $t^*>0$ by the preceding argument. Suppose $t^*\neq T$. Then there exists $\kappa>0$ such that $t^*<t^*+\kappa<T$ and $h(t)=c\neq 0$ when $t^*<t<t^*+\kappa$. Without loss of generality,  we may assume  $(t^*,t^*+\kappa)\cap (I_{g}\cup I_{\widetilde{g}})=\emptyset$. Since $p=\widetilde{p}$ and $h=0$ when $0<t<t^*$, from the flux representation \eqref{eqn:flux-rep}, we have
    \[c\int_{t^*}^t\p_{\nu}K_{\Omega}(x,p,t)\mathrm{d}t=0,\quad \forall t \in (t^*,t^*+\kappa).\]
Repeating the preceding argument yields
\[ c\sum_{\lambda_n=\lambda_j}\p_{\nu}\varphi_n(z_{\ell})\varphi_n(p)=0,\quad \ell=1,\ldots,d,
\]
which implies $\varphi_n(p)=j_0(\alpha_{0l}r)=0$ for these $\lambda_n=\alpha_{0l}^2$. Thus we find a root $\alpha_{01}r$ which is a positive root of $j_0$ and is smaller that $\alpha_{01}$. Since $j_0(0)\neq 0$, this is impossible. This contradiction implies $t^*=T$, and thus $g=\widetilde g$ holds for $0\leq t\leq T$.
\end{proof}

\subsection{Proof of Theorem \ref{thm:compact}}
The proof employs the Laplace transform $\mathcal{L}u(x,s)$ of $u$, defined by $\mathcal{L}u(x,s)=\int_0^\infty e^{-st}u(x,t)\d t$. Below we let  $\mathbb{C}_+=\{s\in\mathbb{C}: \Re(s)\geq0\}.$
\begin{lemma}\label{lem:laplace}
Under the condition of Theorem \ref{thm:compact}, the Laplace transform $\mathcal{L}u(x,s)$ of the solution $u$ of problem \eqref{eqn:heat} with the source $g(t)\delta(x-p)$ exists for all $s\in \mathbb{C}_+$ and is given by
    \[\mathcal{L}u(x,s)=\mathcal{L}g(s)\mathcal{L}K_{\Omega}(x,p,s),\quad  x\in \Omega_\varepsilon,\]
for any $\varepsilon>0$.
    Moreover, we have
    \[\mathcal{L}(\p_{\nu}u)(z,s)=\mathcal{L}g(s)\mathcal{L}(\p_{\nu}K_{\Omega})(z,p,s),\quad z\in\partial\Omega.\]
\end{lemma}
\begin{proof}
By Lemma \ref{heat_kernel}, we have the following upper bound:
    \[|K_{\Omega}(x,y,t)|\leq Ct^{-\frac{d}{2}}e^{-C'\frac{|x-y|^2}{t}},\quad (x,t)\in\Omega\times(0,+\infty)\]
    for some positive constants $C$ and $C'$. Thus for $x\in \Omega_{\varepsilon}$, we have
\[
    |K_{\Omega}(x,p,t)|\leq Ct^{-\frac{d}{2}}e^{-\frac{C_{\varepsilon}}{t}},\quad t\in(0,+\infty).
\]
Also, by Lemma \ref{lem:control}, when $t>0$, we have
\[
|K_{\Omega}(x,p,t)|\leq \sum_{n= 1}^\infty e^{-\lambda_nt}|\varphi_n(x)\varphi_n(p)|\leq Ce^{-\lambda_1t},\quad x\in\Omega_\varepsilon.
\]
Thus $K_\Omega(x,p,\cdot)\in L^1(0,\infty)$ uniformly for $x\in \Omega_\varepsilon$. Since $g\in L^1(0,T)$, by Young's inequality for convolution, we deduce that
\[
\|u(x,\cdot)\|_{L^1(0,\infty)}\leq \|g\|_{L^1(0,T_1)}\|K_{\Omega}(x,p,\cdot)\|_{L^1(0,\infty)}
\]
holds uniformly in $x\in \Omega_\varepsilon$. Thus the Laplace transform $\mathcal{L}u(x,s)$ of $u$ exists for any $s\in \mathbb{C}_+$ and is given by
\[
\mathcal{L}u(x,s)=\mathcal{L}g(s)\mathcal{L}K_{\Omega}(x,p,s),\quad x\in\Omega_\varepsilon.
\]
Meanwhile, we have $\p_{\nu}K\in C^{\infty}(\p\Omega\times[0,\infty))$ and $|\p_{\nu}K_{\Omega}(x,y,t)|\leq Ce^{-\lambda_1t}$ for $t>0$. Thus, we can still derive $\p_{\nu}u(x,\cdot)\in L^1(0,\infty)$ uniformly in $x\in\partial\Omega$. Thus there holds
\[ \p_{\nu}u(z,t)=\int_0^tg(t-s)\p_{\nu}K_{\Omega}(z,p,s)\mathrm{d}s,\quad (z,t)\in\partial\Omega\times(0,+\infty),
\]
and by taking Laplace transform, we obtain the desired result.
\end{proof}

Now we can state the proof of Theorem \ref{thm:compact}.
\begin{proof}
We prove the two cases (i) and (ii) separately.\\
(i) When $g=\widetilde{g}\in \mathcal{A}_c$ is known, by Lemma \ref{lem:laplace}, for $z\in\partial\Omega$ and $s>0$, we have
\begin{align*}
\mathcal{L}(\p_{\nu}u)(z,s)=\mathcal{L}g(s)\mathcal{L}(\p_{\nu}K_{\Omega})(z,p,s)\quad
\mbox{and}\quad
\mathcal{L}(\p_{\nu}\widetilde{u})(z,s)=\mathcal{L}g(s)\mathcal{L}(\p_{\nu}K_{\Omega})(z,\widetilde{p},s).
\end{align*}
Fix $\ell\in\{1,2\}$. By Lemma \ref{lem:ana-ext}, the map $t\mapsto \partial_\nu u(z_\ell,t)$ is analytic on $(T_1,+\infty)$. Thus,  the isolated zero theorem implies that $\p_{\nu}u(z_{\ell},t)=\p_{\nu}\widetilde{u}(z_{\ell},t)$ for any $t>0$. By substituting $z=z_{\ell}$,  we obtain $\mathcal{L}(\p_{\nu}u)(z_{\ell},s)=\mathcal{L}(\p_{\nu}\widetilde{u})(z_{\ell},s)$, $s>0$. This directly yields
\[\mathcal{L}g(s)\mathcal{L}(\p_{\nu}K_{\Omega})(z_{\ell},p,s)=\mathcal{L}g(s)\mathcal{L}(\p_{\nu}K_{\Omega})(z_{\ell},\widetilde{p},s),\quad \ell=1,2,\ s>0.\]
Since $g\in L^1(0,T)$ is compactly supported and non-uniformly vanishing, the Laplace transform $\mathcal{L}g$ is holomorphic  in $ \mathbb{C}$ and non-uniformly vanishing. Then, the isolated zero theorem implies that there exists $0<\tau_1<\tau_2$ such that $\mathcal{L}g(s)\neq0$, $s\in(\tau_1,\tau_2)$. This implies
$$\mathcal{L}(\p_{\nu}K_{\Omega})(z_{\ell},p,s)=\mathcal{L}(\p_{\nu}K_{\Omega})(z_{\ell},\widetilde{p},s),
\quad \ell=1,2,\ s\in(\tau_1,\tau_2).$$
since the maps $s\mapsto \mathcal{L}(\p_{\nu}K_{\Omega})(z_{\ell},p,s)$, $s\mapsto\mathcal{L}(\p_{\nu}K_{\Omega})(z_{\ell},\widetilde{p},s)$ are analytic with respect to $s\in(0,+\infty)$ and by the isolated zero theorem, we obtain
\begin{equation}\label{eqn-flux-Lap}
\mathcal{L}(\p_{\nu}K_{\Omega})(z_{\ell},p,s)=\mathcal{L}(\p_{\nu}K_{\Omega})(z_{\ell},\widetilde{p},s),\quad \ell=1,2,\ s>0.
\end{equation}
 Thus we have for almost every $t>0$,
\[
\p_{\nu}K_{\Omega}(z_{\ell},p,t)=\p_{\nu}K_{\Omega}(z_{\ell},\widetilde{p},t),\quad \ell=1,2.
\]
Or equivalently,
\[
\sum_{n= 1}^\infty e^{-\lambda_nt}\p_{\nu}\varphi_n(z_{\ell})(\varphi_n(p)-\varphi_n(\widetilde{p}))=0
,\quad \ell=1,2.\]
Comparing this identity with \eqref{eq2} and repeating the argument in the proof  of Theorem \ref{thm:ball} yield the assertion $p=\widetilde{p}$.

\noindent(ii)   Fix $\ell\in\{1,2,3\}$.
Since $\p_{\nu}u(z_{\ell},t)=\p_{\nu}\widetilde{u}(z_{\ell},t)$ for $0< t\leq T$, by the time analyticity  of Lemma \ref{lem:ana-ext}, the identity actually holds for any $ t\in(0,+\infty)$. Thus by Lemma \ref{lem:laplace}, setting $s=0$ in the identity \eqref{eqn-flux-Lap} leads to
    \[\mathcal{L}(\p_{\nu}u)(z_{\ell},0)=\mathcal{L}g(0)\p_{\nu}(\mathcal{L}K_{\Omega})(z_{\ell},p,0),\quad \ell=1,2,3.\]
By the definition of Laplace transform and using Lemma \ref{Green_heat}, the identity implies
\[
\mathcal{L}(\p_{\nu}u)(z_{\ell},0)=\mathcal{L}g(0)\p_{\nu}\left(\int_0^{\infty}K_{\Omega}(t,x,p)\mathrm{d}t\right)=\mathcal{L}g(0)P(p,z_{\ell}),\quad \ell=1,2,3.
\]
Similarly, we have
\begin{equation*}
\mathcal{L}(\p_{\nu}\widetilde{u})(z_{\ell},0)=\mathcal{L}\widetilde{g}(0)P(\widetilde{p},z_{\ell}),\quad \ell=1,2,3.
\end{equation*}
Consequently,
\[
\mathcal{L}g(0)P(p,z_{\ell})=\mathcal{L}\widetilde{g}(0)P(\widetilde{p},z_{\ell}),\quad \ell=1,2,3.
\]
Note that $g,\widetilde{g}\in\mathcal{A}_c$ implies $\mathcal{L}g(0),\mathcal{L}\widetilde{g}(0)\neq 0$.
Since the domain $\Omega\subset \RR^2$ is bounded and simply connected with a smooth boundary $\partial\Omega$, by Kellogg-Warschawski theorem \cite[Theorem 3.5]{Pommerenke1992}, there exists a Riemann mapping $f\colon\Omega\to D$ that can be smoothly extended to a $C^1$ diffeomorphism $F\colon\overline{\Omega}\to \overline{D}$, where $D$ is the open unit disc in $\RR^2$ such that $F'(z)\neq 0$ for any $z\in \overline{\Omega}$ and $F(\p\Omega)=\p D$. The Poisson kernels $P$ on $\Omega$ and $P_0$ on $D$ satisfy the following relation:
    \[P(x,z)=P_0(F(x),F(z))\cdot |F'(z)|.\]
Recall that $ P_0(x,z)=\frac{1}{2\pi}\frac{1-|x|^2}{|x-z|^2}$.
    Thus upon letting $a=\mathcal{L}g(0)$ and $b=\mathcal{L}\widetilde{g}(0)$, we obtain
    \[a\frac{1-|F(p)|^2}{|F(p)-F(z_{\ell})|^2}=aP(p,z_{\ell})=bP(\widetilde{p},z_{\ell})=b\frac{1-|F(\widetilde{p})|^2}{|F(\widetilde{p})-F(z_{\ell})|^2}.\]
    This directly implies
    \[\frac{a(1-|F(p)|^2)}{b(1-|F(\widetilde{p})|^2)}=\frac{|F(p)-F(z_{\ell})|^2}{|F(\widetilde{p})-F(z_{\ell})|^2}.\]
    Now suppose $p\neq \widetilde{p}$. Then $F(p)\neq F(\widetilde{p})$. Consider the set
    $$A_k=\{F(z)\in \RR^2:|F(p)-F(z)|=k|F(\widetilde{p})-F(z)|\}.$$
    Let $k=\frac{a(1-|F(p)|^2)}{b(1-|F(\widetilde{p})|^2)}.$
Then all points $F(z_{\ell})$, $ \ell=1,2,3$, should be lying on $\partial D\cap A_k$. When $k>0, k\neq 1$, the set $A_k$ is a circle. When $k=1$, the set $A_k$ is a line. In both cases, we have either  Card$(\partial D\cap A_k)\leq 2$ (with Card denoting taking the cardinality of a set) or $k\neq1$ and $A_k=\partial D$. Since Card$(\{F(z_{\ell}):\ \ell=1,2,3\})=3>2$, we deduce that $k\neq1$ and $A_k=\partial D$ which is impossible when $F(p)\neq F(\widetilde{p})$. Thus, we obtain $p=\widetilde{p}$. In addition, we get
    \[\mathcal{L}g(s)\mathcal{L}(\p_{\nu}K_{\Omega})(z_{\ell},p,s)=\mathcal{L}\widetilde{g}(s)\mathcal{L}(\p_{\nu}K_{\Omega})(z_{\ell},p,s),\quad \ell=1,2,3,\ s>0.
    \]
Now we prove $\mathcal{L}(\p_{\nu}K_{\Omega})(z,p,s)<0$ for any $z\in \p\Omega$ and $s\geq 0$, which will imply $\mathcal{L}g(s)=\mathcal{L}\widetilde{g}(s)$ and thus $g=\widetilde{g}$ by the uniqueness of Laplace transform.  In fact, by Lemma \ref{Green_heat}, we have
\[
\mathcal{L}(\p_{\nu}K_{\Omega})(z,p,s)=\p_{\nu}\left(\int_{0}^{\infty}e^{-s t}K_{\Omega}(z,p,t)\mathrm{d}t\right)=\p_{\nu}G_{s}(z,p),
\]
with the function $G_s$ defined in \eqref{eqn:G-lam}.
By the expression of $G_s$, we derive $G_s(x,y)>0$ when $x,y\in \Omega$. Fix an $0<\varepsilon<\mathrm{dist}(p,\p\Omega)$ and choose $\zeta\geq0$ to be a smooth cut-off function such that $\zeta(x)=1$ when $|x-p|>\varepsilon$. Then the function $w(x)=\zeta(x)G_s(x,p)$ solves
\[
\left\{\begin{aligned}
(sI-\Delta)w&=0,\quad \mbox{in }\Omega\backslash B(p,\varepsilon),\\
w&=0,\quad \mbox{on } \p\Omega,\\
w&=\zeta(x)G_s(x,p),\quad \mbox{on }\p B(p,\varepsilon).
\end{aligned}\right.
\]
By Hopf's lemma and the fact $\p_{\nu}w(z)=\p_{\nu}G(z,p)$ for $z\in \p\Omega$, we deduce the desired assertion.
\end{proof}

\begin{remark}\label{rmk:compact}
The argument in case (ii) can also be extended to $\mathbb{R}^d$ when $\Omega\subset \mathbb{R}^d$ is the unit ball. Note that the Poisson kernel $P(x,z)$ takes the following form:
\[
P(x,z)=\frac{1}{\omega_{d-1}}\frac{1-|x|^2}{|x-z|^d}.
\]
Let
$A_k=\{z\in \mathbb{R}^d:|p-z|=k|\widetilde{p}-z|\}$.
With $(\cdot,\cdot)$ being the standard inner product on $\mathbb{R}^d$, the points in $A_k$ satisfy
$
(p-z,p-z)=k^2(\widetilde{p}-z,\widetilde{p}-z)$,
or equivalently,
\[
(1-k^2)|z|^2-(z,2p-2k^2\widetilde{p})+|p|^2-k^2|\widetilde{p}|^2=0.
\]
When $z\in \p\Omega$, we obtain
    \[A_k\cap \p\Omega=\{z:|z|=1,~(z,2p-2k^2\widetilde{p})=1+|p|^2-k^2(1+|\widetilde{p}|^2)\}.\]
    That is, $A_k\cap \p\Omega$ is the intersection of $\p\Omega$ and the hyperplane \[(z,2p-2k^2\widetilde{p})-(1+|p|^2-k^2(1+|\widetilde{p}|^2))=0.\]
Thus if we take $z_{\ell}$ such that $\ell=d+1$ and $z_{\ell}$ do not lie on the same hyperplane, we get a contradiction and conclude $p=\widetilde{p}$.
\end{remark}

\section{Numerical experiments and discussions}
\label{sec:numer}

Now we present several numerical experiments in the two-dimensional case to illustrate the feasibility of recovering the location of the point source and the time-dependent amplitude $g(t)$ from sparse boundary measurement. Throughout, we fix $T=1$. The direct problems are solved using \texttt{MATLAB} PDE Toolbox, which employs a standard Galerkin finite element method (with conforming piecewise linear elements) in space and an implicit time-stepping scheme in time \cite{Thomee:2006}. To represent the point source $\delta(x-p)$ on the finite element mesh, we employ a Gaussian approximation of the form
$\delta(x-p)\approx \frac{1}{2\pi\sigma^2}e^{-\frac{|x-p|^2}{2\sigma^2}}$,
with $\sigma=0.03$. We employ a finer space--time mesh (with a mesh size $h_f=0.02$ and 600 time steps) to generate the exact flux $\partial_\nu u(z_\ell,t)$, and a coarse mesh (with a mesh size $h_c=0.04$ and $300$ time steps) for the  reconstruction, so as to mitigate the so-called ``inverse crime''. We generate the noisy data $y^\delta$ by
\begin{equation}\label{eq:noise-model}
y^\delta(z_\ell,t_k)
= \partial_\nu u(z_\ell,t_k)\,\bigl(1 + \delta\,\xi_{k,\ell}\bigr),\quad\mbox{with }
\xi_{k,\ell}\overset{\mathrm{i.i.d.}}{\sim}\mathcal N(0,1),
\end{equation}
where $\delta>0$ denotes the relative noise level. The notation $(r,\theta)$ denotes the polar coordinate, and $(x_1,x_2)$ denotes the Cartesian coordinates.

We employ an alternating minimization algorithm for the numerical reconstruction. In the case of a constant amplitude, to update the source location $(x_1,x_2)$, we employ the regularized Gauss-Newton method, with the Jacobian approximated by the central finite difference (with $h_{x_1}=h_{x_2}=10^{-3}$) and a damping parameter $\lambda=10^{-4}$. We determine the step length by the Armijo's rule \cite[Chapter 6]{DennisSchnabel:1996} (initial step size $\alpha=1$, contraction factor $\rho=0.5$, minimum step size $\alpha_{\min}=1/16$, Armijo constant $c=10^{-4}$). For the amplitude $g$, we take an analytic update using the standard least-squares. More precisely, let $q=(x_1,x_2)$ denote the current location, and let $Y(q)\in\mathbb{R}^{(N_t+1)\times L}$ be the flux traces at all $N_t+1$  time steps and $L$ sensors when the forward problem is solved with unit strength $g\equiv1$. The optimal amplitude at $q$ is given by
$g(q) \;=\; {\langle Y(q),y^\delta\rangle}/{\|Y(q)\|_2^2}$. In the numerical experiments, the algorithm  converges rapidly  (often within five iterations). The adaptation of the  algorithm to other cases is direct. The \texttt{MATLAB} code is available at \url{https://github.com/fangyugong/point-source-identification-heat}.

The first example is to recover the location $p$ and the constant magnitude $g$ on the unit disc.
\begin{example}\label{exam:const-g}
The domain $\Omega$ is the unit disc, and the point source is located at $(r^\dag,\theta^\dag) = (0.4, 2.0)$, i.e., $(x_1^\dag,x_2^\dag) \approx (-0.1665, 0.3637)$, with a constant amplitude of $g = 2.0$. The flux is observed at two boundary points $z_\ell = e^{\mathrm{i}\theta_\ell}$ with $\theta_\ell \in \{1.7,\;2.0\}$.
\end{example}

The scheme is initialized to $(r_0, \theta_0) = (0.5, 1.8)$. We present numerical results for three noise levels (3\%, 5\%, and 10\%) in Table~\ref{tab:const-g}.
The algorithm can successfully recover the location $p$ of the source and the amplitude $g$, and as the noise level $\delta$ increases from 3\% to 10\%, the error in the estimated parameters also increases. Nonetheless, even with 10\% noise, the estimates remain close to the exact ones, showing the robustness of the method.

\begin{table}[hbt!]
    \centering
\begin{threeparttable}
    \caption{The recovery of the location $(x_1,x_2)$ and the amplitude $g$ for Example \ref{exam:const-g} with 3\%, 5\% and 10\% noise.}
    \label{tab:const-g}
    \begin{tabular}{c c c c c}
        \toprule
           & & {$x_1$} & {$x_2$} & {$g$} \\
        \midrule
        & exact & -0.1665 & 0.3637 & 2.0000 \\
        \midrule
       \multirow{ 2}{*}{3\%} & estimate & -0.1655 & 0.3645 & 1.9979 \\
       & error               & \num{9.80e-4} & \num{8.11e-4} & \num{2.13e-3} \\
        \midrule
       \multirow{ 2}{*}{5\%} & estimate & -0.1647 & 0.3648 & 1.9968 \\
       & error               & \num{1.77e-3} & \num{1.12e-3} & \num{3.24e-3} \\
        \midrule
       \multirow{ 2}{*}{10\%} & estimate & -0.1627 & 0.3657 & 1.9939 \\
       & error               & \num{3.74e-3} & \num{1.94e-3} & \num{6.11e-3} \\
        \bottomrule
    \end{tabular}
\end{threeparttable}
\end{table}

The next example is about recovering the location $(x_1,x_2)$ only, with a known time-dependent amplitude $g(t)$.
\begin{example}\label{exam:known-g}
The time-dependent amplitude $g(t) = \sin(2\pi t) + 1$ is known, and only the location $p=(r\cos\theta,r\sin\theta)$ is unknown. Consider the following two settings.
\begin{itemize}
\item[{\rm(a)}] The domain $\Omega$ is the unit disc, and there is one point source at
$(r^\dag,\theta^\dag)=(0.4,\,2.0)$, i.e., $p^\dag=(x_1^\dag,x_2^\dag)\approx(-0.1665,\,0.3637)$.
The flux is observed at two boundary points $z_\ell=e^{\mathrm{i}\theta_\ell}$ with  $\theta_\ell\in \{1.7,\;2.0\}$.
\item[{\rm(b)}]  The domain $\Omega$ is an ellipse
$\{(x_1/a)^2+(x_2/b)^2\le 1\}$ with $(a,b)=(1.2,\,0.8)$, and there is one point source at
$(r^\dag,\theta^\dag)=(0.4,\,2.0)$, i.e.,
$p^\dag=(a r^\dag\cos\theta^\dag,\ b r^\dag\sin\theta^\dag)\approx(-0.1998,\,0.2910)$.
The flux is measured at two boundary points
with angles $\theta_\ell\in \{1.7,\;2.0\}$.
\end{itemize}
\end{example}

Table \ref{tab:known-g} presents the exact and estimated locations for Example \ref{exam:known-g} at three noise levels ($\delta\in\{3\%,5\%,10\%\}$). The reconstructions for both cases (a) and (b) for up to 10\% noise in the data are fairly accurate, agreeing with the uniqueness result in Theorems \ref{thm:ball} and \ref{thm:compact}(b). Note that the accuracy of the reconstruction in cases (a) and (b) are largely comparable, and it is also worth noting that case (b) employs only two points, instead of three points as alluded by Theorem \ref{thm:compact}(ii), indicating the potential for further reducing the amount of the data.

\begin{table}[hbt!]
    \centering
\begin{threeparttable}
\caption{The recovery of the location $(x_1,x_2)$ for Example \ref{exam:known-g} with 3\%, 5\% and 10\% noise.\label{tab:known-g}}
    \begin{tabular}{c c c ccc}
        \toprule
       \multicolumn{2}{c}{} & \multicolumn{2}{c}{case (a)} &
      \multicolumn{2}{c}{case (b)} \\
      \cline{3-4} \cline{5-6}
       &  & {$x_1$} & {$x_2$} & $x_1$ & $x_2$\\
        \midrule
       &  exact & -0.1665 & 0.3637 &  -0.1998 & 0.2910 \\
        \midrule
       \multirow{ 2}{*}{3\%}
       &  estimate & -0.1673 & 0.3635 & -0.2002 & 0.2907 \\
       & error     & \num{8.50e-4} & \num{1.82e-4} & \num{4.03e-4} & \num{3.15e-4}\\
        \midrule
       \multirow{ 2}{*}{5\%}
       & estimate  & -0.1678 & 0.3633 & -0.2003 & 0.2904 \\
       & error     & \num{1.30e-3} & \num{4.68e-4} & \num{5.83e-4} & \num{6.05e-4} \\
        \midrule
      \multirow{ 2}{*}{10\%}
       & estimate  & -0.1687 & 0.3626 & -0.2008 & 0.2896 \\
       & error     & \num{2.27e-3} & \num{1.16e-3} & \num{1.04e-3} & \num{1.33e-3}\\
        \bottomrule
    \end{tabular}
\end{threeparttable}
\end{table}

Next we simultaneously recover the location and a piecewise-constant time-dependent part.
\begin{example}\label{exam:pwc-g}
The domain $\Omega$ is the unit disc, and the source is located at $(x_1^\dag,x_2^\dag)=(-0.1665,\,0.3637)$ with a piecewise-constant amplitude
$g(t)=c_1\,\chi_{t\le T_1}+c_2\,\chi_{t>T_1}$, with
$(c_{1}^\dag,c_{2}^\dag,T_{1}^\dag)
=(2.0,\,1.0,\,0.4)$. The flux is collected at
two angles $\theta_\ell\in\{1.7,\,2.0\}$.
\end{example}

The reconstruction is again based on alternating minimization. For any fixed
$q=(x_1,x_2,T_1)$, we update
$c^\star=(c_1,c_2)$ using the least-squares fitting,  and update the  variables $q$ by a regularized Gauss-Newton step with
Armijo-type backtracking.
We employ an adaptive step to approximate the derivative with respect to $T_1$. Specifically, let $\Delta t$ be the coarse time step and $\eta=\min\{T_1-t_j,\;t_{j+1}-T_1\}$ the distance from $T_1$
to the nearest grid line (with $t_j=j\Delta t$), and take the step size
$h_{T_1} \;=\; \max\bigl(1.25\,\eta,\; 2\Delta t,\; h_{\min}\bigr)$, with $h_{\min}=5\times10^{-3}$,
and set $T_1^\pm=T_1\pm h_{T_1}$ (clamped to $[0,T]$).
We investigate the example at three noise levels $\delta\in\{1\%,3\%,5\%\}$.
The boundary flux traces in Fig. \ref{fig:pwc-g} shows the monotone decrease of the misfit as the noise level $\delta$ decreases.
The final estimates and absolute
errors are given in Table \ref{tab:pwc-g}. The accuracy of the estimates is comparable with that for Examples \ref{exam:const-g} and \ref{exam:known-g}. These observations indicate the stable recovery of a piecewise constant amplitude $g(t)$ and its location $p$.

\begin{table}[hbt!]
\centering
\begin{threeparttable}
\caption{The numerical results for Example \ref{exam:pwc-g} with $1\%$, $3\%$ and $5\%$ noise.\label{tab:pwc-g}}
\begin{tabular}{cc c c c c c}
\toprule
& & {$x_1$} & {$x_2$} & {$T_1$} & {$c_1$} & {$c_2$} \\
\midrule
& exact & -0.1665 & 0.3637 & 0.4000 & 2.0000 & 1.0000 \\
\midrule
\multirow{ 2}{*}{1\%} & estimate & -0.1666 & 0.3646 & 0.4001 & 1.9954 & 0.9989 \\
& error                & \num{1.16e-4} & \num{8.93e-4} & \num{8.83e-5} & \num{4.63e-3} & \num{1.12e-3} \\
\midrule
\multirow{ 2}{*}{3\%} & estimate & -0.1665 & 0.3656 & 0.4003 & 1.9888 & 0.9976 \\
& error                & \num{6.15e-6} & \num{1.84e-3} & \num{2.67e-4} & \num{1.12e-2} & \num{2.37e-3} \\
\midrule
\multirow{ 2}{*}{5\%} & estimate & -0.1664 & 0.3666 & 0.4006 & 1.9815 & 0.9961 \\
& error                & \num{9.77e-5} & \num{2.88e-3} & \num{5.93e-4} & \num{1.85e-2} & \num{3.91e-3} \\
\bottomrule
\end{tabular}
\end{threeparttable}
\end{table}
\begin{figure}[hbt!]
\centering
\setlength{\tabcolsep}{0pt}
\begin{tabular}{@{}c@{\hspace{1.5em}}c@{}}
\includegraphics[width=0.42\textwidth]{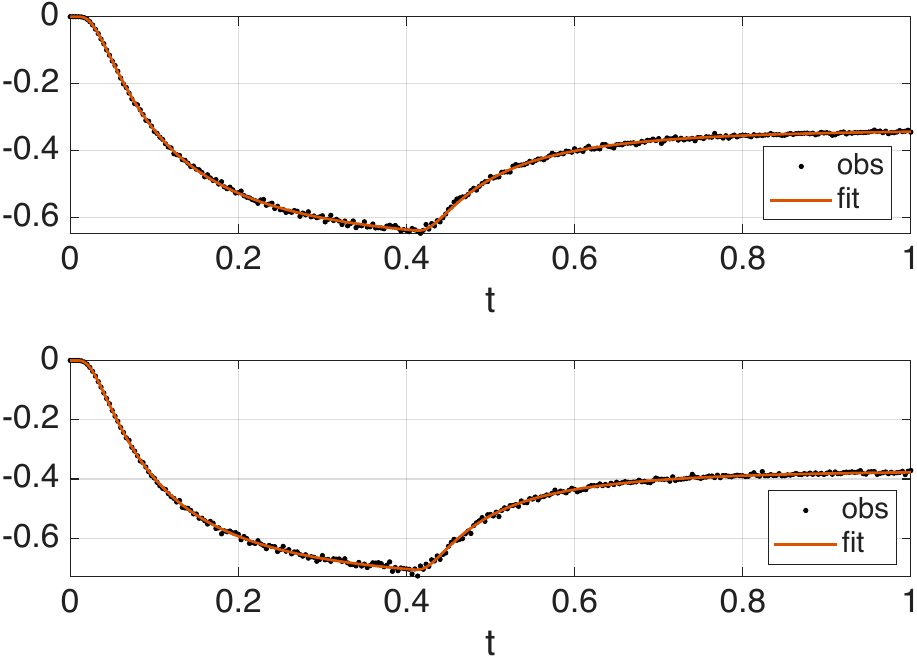}
&\includegraphics[width=0.3\textwidth]{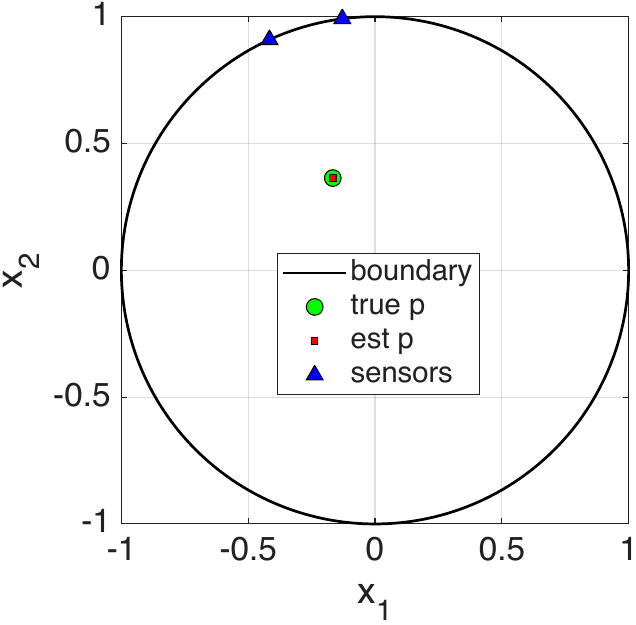}\\[1.5em] \includegraphics[width=0.42\textwidth]{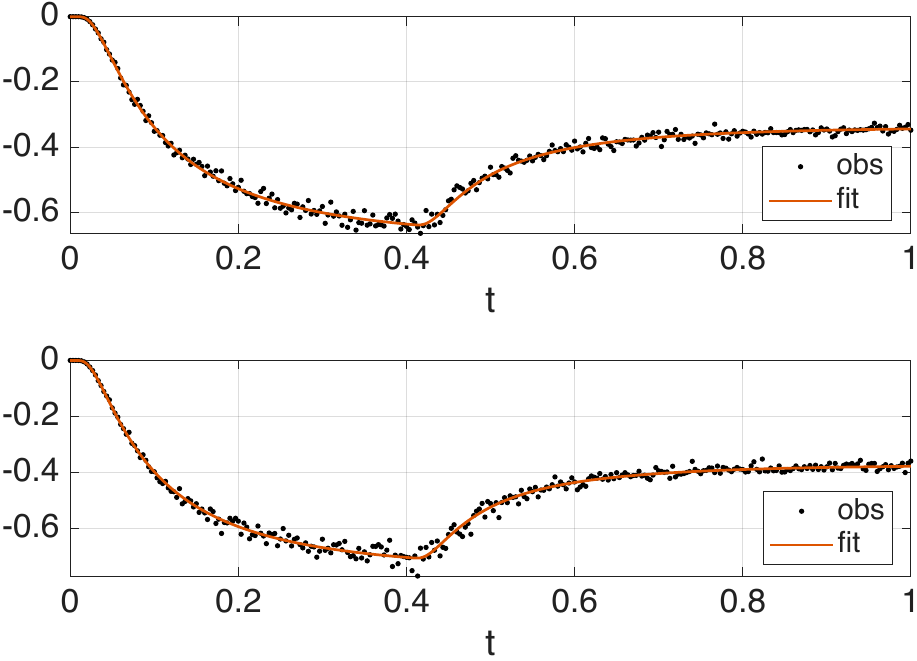}
&\includegraphics[width=0.3\textwidth]{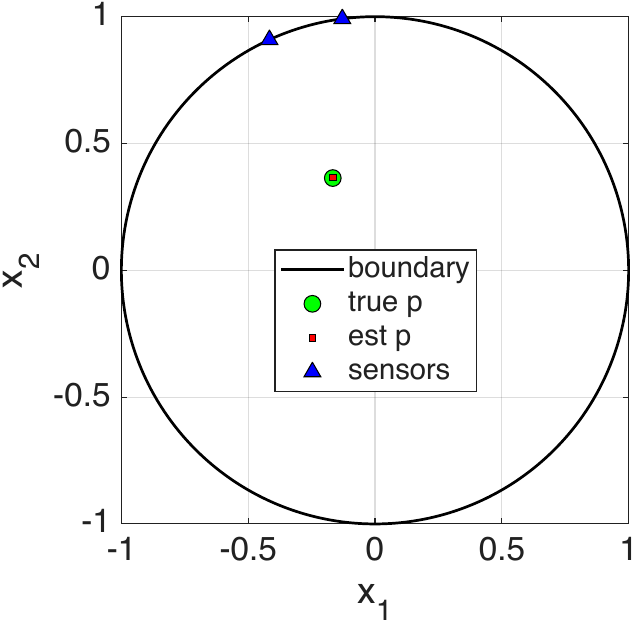}\\[1.5em]
\includegraphics[width=0.42\textwidth]{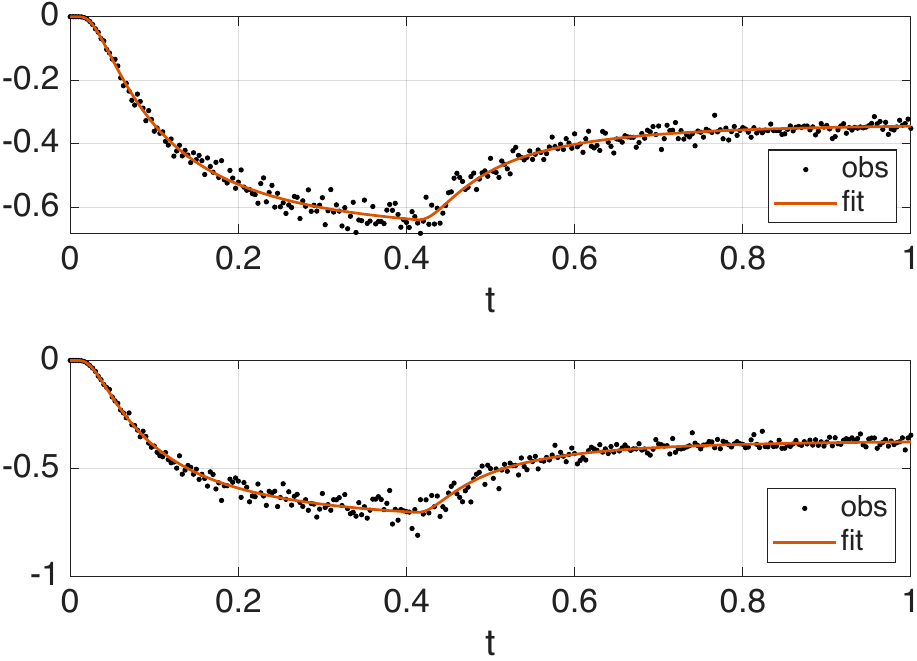}&
\includegraphics[width=0.3\textwidth]{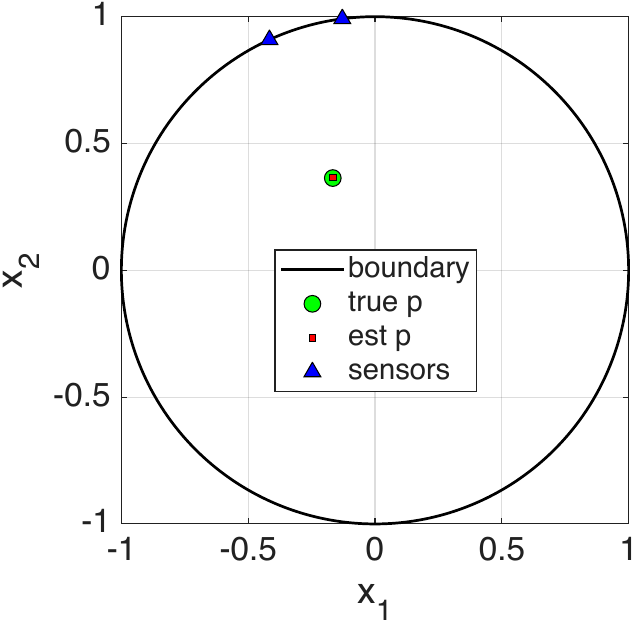}\\
flux trace & recovered location
\end{tabular}
\caption{The flux traces and recovered location for Example \ref{exam:pwc-g} with $\delta=1\%$ (top), $\delta =3\%$ (middle) and $\delta=5\%$ (bottom).}
\label{fig:pwc-g}
\end{figure}

The last example is about recovering the location and time-dependent part $g(t)$ on an ellipse.
\begin{example}\label{exam:cpt-g}
Let \(
\Omega=\{(x_1/a)^2+(x_2/b)^2\le 1\}
\)
with $(a,b)=(1.2,0.8)$, and the source $F(x,t)= g(t)\,\delta(x-p)$,
with an unknown location $p$ and time-dependent amplitude $g(t)$. The ground truth $(r^\dag ,\theta^\dag )=(0.4,\,2.0)$, and a compactly supported Hann window (``half-sine squared'')
$g^\dag (t)=2\mathbf{1}_{[0,T_1]}(t)\,\sin^2\!\bigl(\pi t/T_1\bigr)$ with $T_1=0.5$. The flux is observed at three boundary points $z_\ell=(a\cos\theta_\ell,\ b\sin\theta_\ell)$ with
$\theta_\ell \in \{\pi/6,\ 1.1\pi/2,\ 5\pi/6\}$.
\end{example}

For the reconstruction, we parameterize the amplitude $g(t)$ as
$g(t) \approx\sum_{k=1}^{K} c_k\,\phi_k(t)$,
with $K$ piecewise-linear (hat) bases $\{\phi_k\}_{k=1}^K$ on the interval $[0,T]$ ($K=20$ in the experiments). Given the location $p=(x_1,x_2)$, we determine the coefficients $(c_k)$ by a least-squares fitting. Given the amplitude $g(t)$,
the location $p=(x_1,x_2)$ is updated using the regularized Gauss-Newton method, with an initial guess $(r_0,\theta_0)=(0.5,\,1.8)$. The numerical results for the example at three noise levels $\delta\in\{1\%,3\%,5\%\}$ are shown in Table \ref{tab:cpt-g} and Fig. \ref{fig:ellipse-cpt-g}.
Table~\ref{tab:cpt-g} reports the estimates and absolute errors for
the location $p=(x_1,x_2)$, and the relative $L^2(0,T)$ error
$e(g)=\| g-g^\dag\|_{L^2(0,T)}/\|g^\dag\|_{L^2(0,T)}$ of the recovered $ g$.
For all three noise levels, the algorithm  converges in a few outer steps with steady reduction of the misfit.
Despite the big challenge associated with estimating the entire time profile $g(t)$, the errors of the recovered locations remain at the level of $10^{-3}$,
and the recovered $g$ attains a small relative $L^2(0,T)$ error between about $1.2\%$ and $3.0\%$.
In Fig.~\ref{fig:ellipse-cpt-g}, we present the flux traces (observations versus fits), recovered location and time-dependent amplitude $g(t)$.

\begin{table}[hbt!]
  \centering
  \begin{threeparttable}
  \caption{The numerical results for Example \ref{exam:cpt-g} with three noise levels $1\%$, $3\%$ and $5\%$.}
  \label{tab:cpt-g}
  \begin{tabular}{cc c c c}
    \toprule
  &  & {$x_1$} & {$x_2$} & {$e(g)$} \\
    \midrule
   & exact & -0.1998 & 0.2910 & --- \\
    \midrule
   \multirow{ 2}{*}{1\%}
    & estimate & -0.1999 & 0.2906 & \num{1.16e-2} \\
    & error    & \num{1.91e-4} & \num{3.28e-4} & --- \\
    \midrule
   \multirow{ 2}{*}{3\%}
    & estimate & -0.2008 & 0.2903 & \num{1.94e-2} \\
    & error    & \num{1.03e-3} & \num{6.41e-4} & --- \\
    \midrule
   \multirow{ 2}{*}{5\%}
    & estimate & -0.2037 & 0.2911 & \num{2.97e-2} \\
    & error    & \num{3.93e-3} & \num{1.58e-4} & --- \\
    \bottomrule
  \end{tabular}
\end{threeparttable}
\end{table}

\begin{figure}[hbt!]
  \centering
\setlength{\tabcolsep}{0pt}
\begin{tabular}{@{}c@{\hspace{.5em}}c@{\hspace{.5em}}c@{}}
\includegraphics[width=0.34\textwidth]{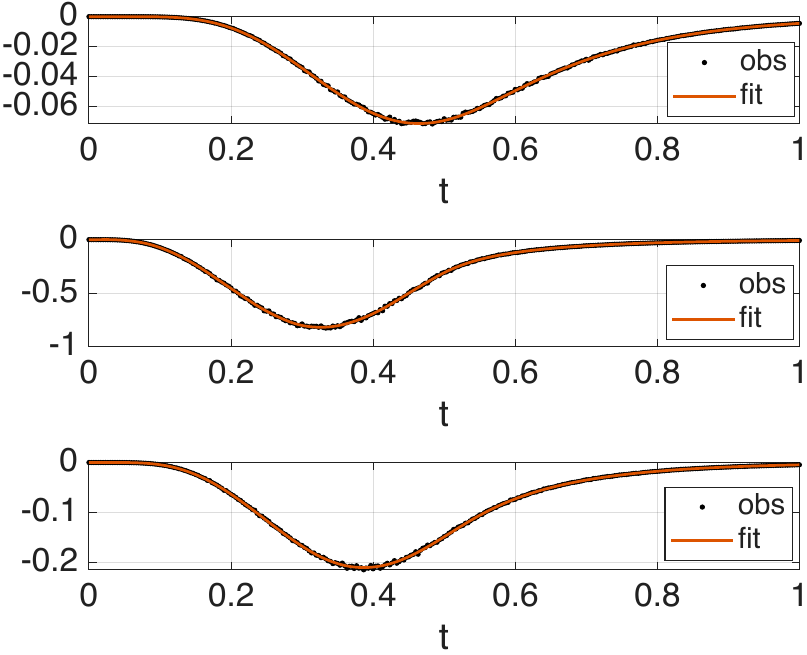}&
\includegraphics[width=0.32\textwidth]{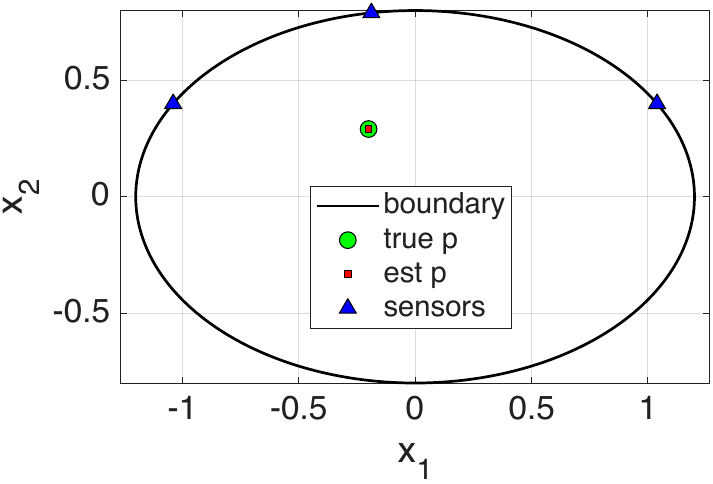}&
\includegraphics[width=0.30\textwidth]{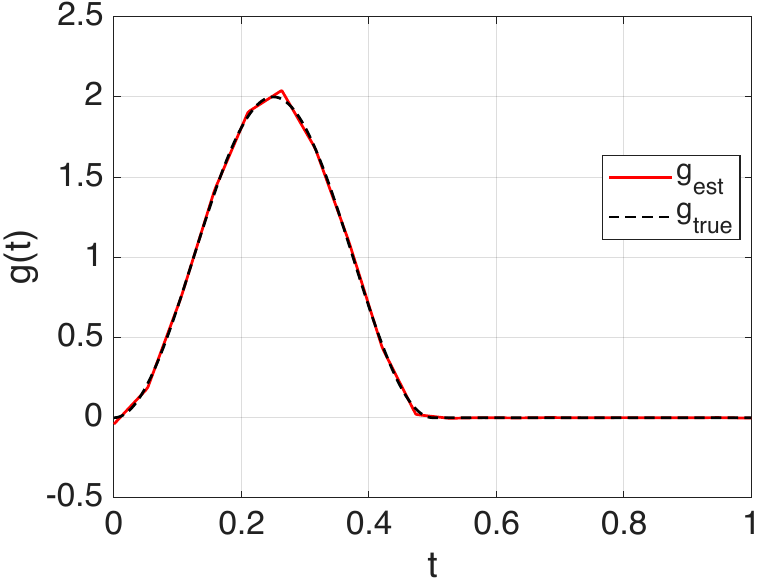}\\[1.5em]
\includegraphics[width=0.34\textwidth]{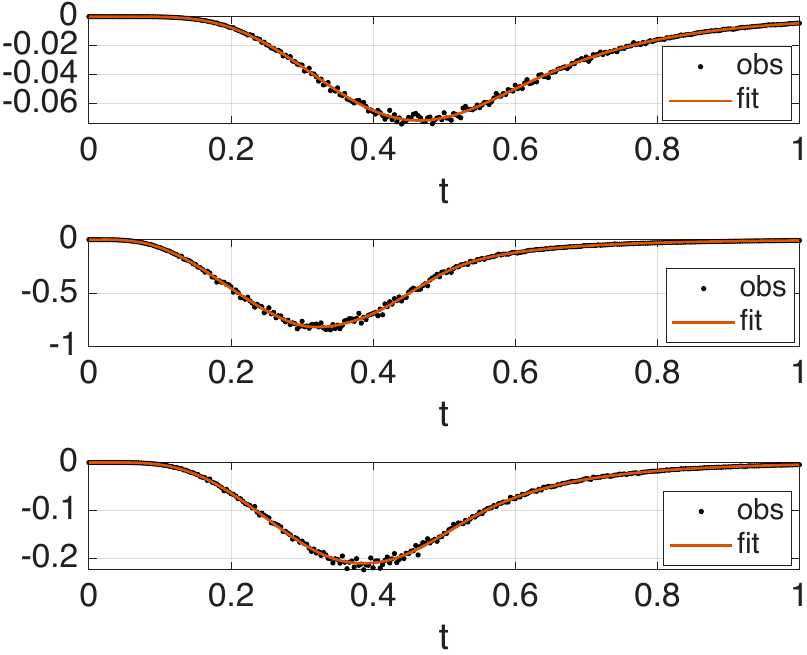}&  \includegraphics[width=0.32\textwidth]{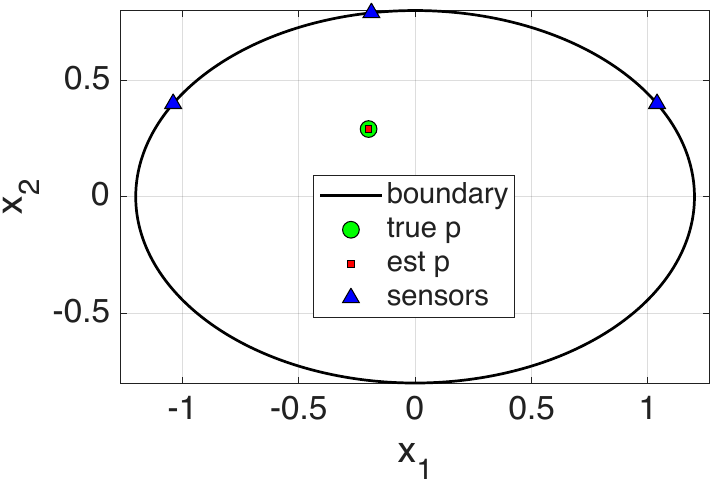}&
\includegraphics[width=0.30\textwidth]{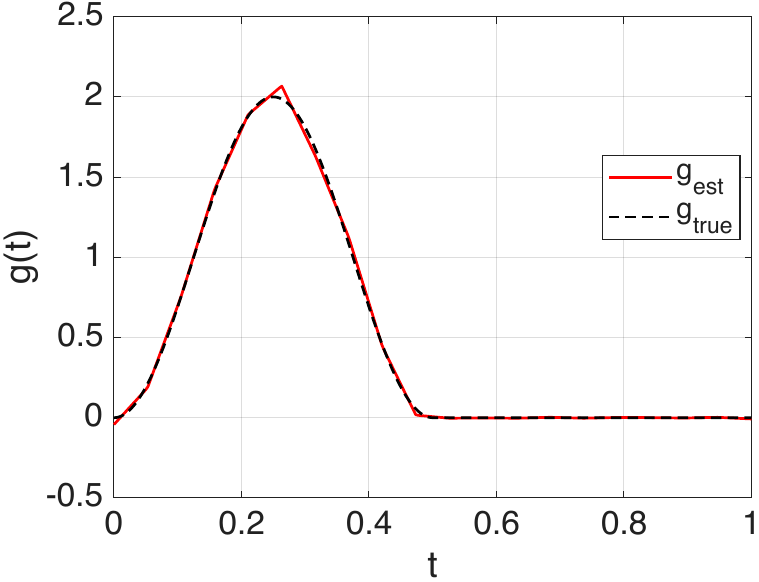}\\[1.5em]
\includegraphics[width=0.34\textwidth]{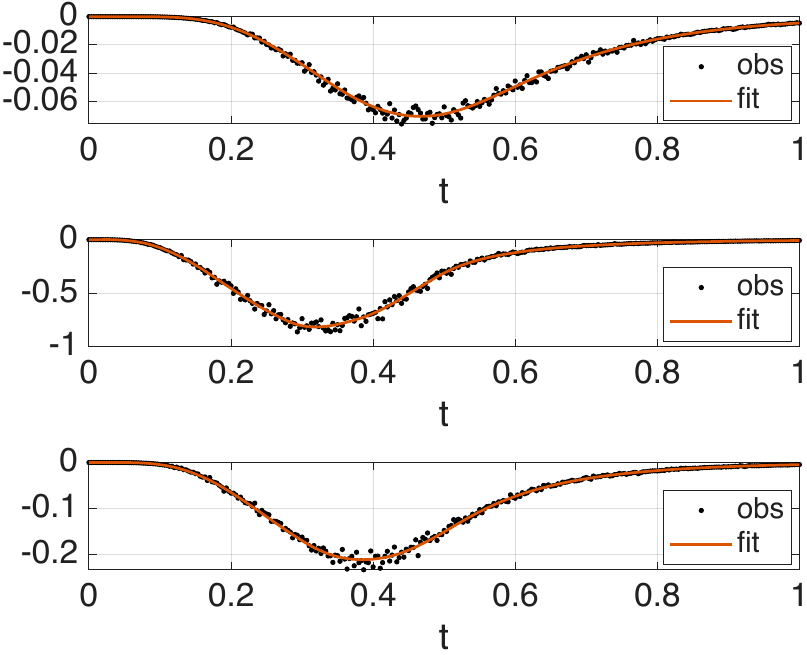}&
\includegraphics[width=0.32\textwidth]{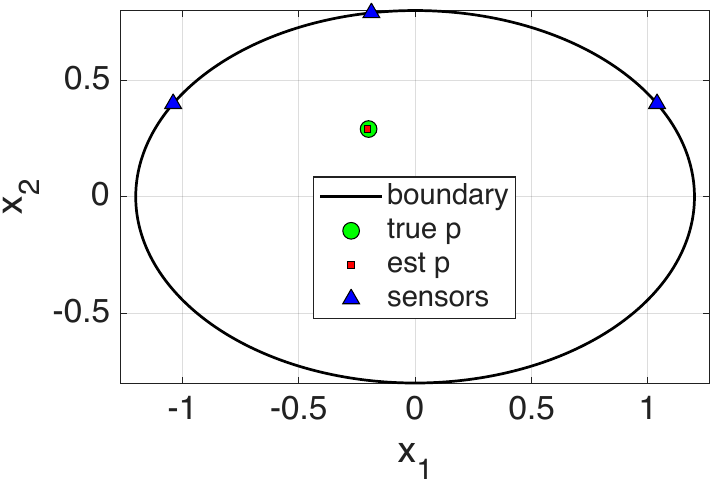}
&\includegraphics[width=0.30\textwidth]{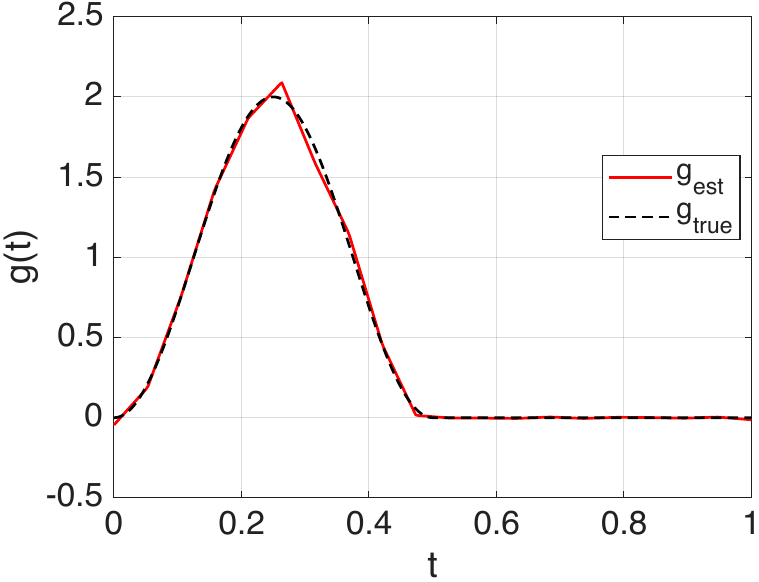}\\
flux traces & location & strength $g(t)$
\end{tabular}
\caption{The numerical results for Example \ref{exam:cpt-g} with $\delta=1\%$ (top), $\delta=3\%$ (middle) and $\delta=5\%$ (bottom).}
\label{fig:ellipse-cpt-g}
\end{figure}

\bibliographystyle{siam}
\bibliography{references}
\end{document}